\documentclass[10pt]{amsart}
\usepackage{graphicx}
\usepackage{amssymb}
\usepackage{enumerate}
\usepackage{url,hyperref}
\usepackage{mathrsfs}
\usepackage{color}

\numberwithin{equation}{section}
\numberwithin{figure}{section}
\theoremstyle{plain}
\newtheorem{thm}{Theorem}[section]
\theoremstyle{definition}
\newtheorem{defn}[thm]{Definition}

\theoremstyle{plain}
\newtheorem{prop}[thm]{Proposition}
\theoremstyle{remark}
\newtheorem{rem}[thm]{Remark}
\theoremstyle{plain}
\newtheorem{cor}[thm]{Corollary}
\theoremstyle{plain}
\newtheorem{lem}[thm]{Lemma}
\theoremstyle{remark}
\newtheorem*{rem*}{Remark}

\newcommand{\abs}[1]{\left\vert#1\right\vert}
\newcommand{\set}[1]{\left\{#1\right\}}
\newcommand{\pset}[1]{\left(#1\right)}
\newcommand{\bset}[1]{\left\langle #1 \right\rangle}
\newcommand{\sqset}[1]{\left[ #1\right]}
\newcommand{\Real}{\mathbb{R}}

\newcommand{\eps}{\varepsilon}

\newcommand{\sph}{\mathbb{S}}
\newcommand{\To}{\rightarrow}
\newcommand{\pd}{\partial}

\newcommand{\PRT}{\mathscr{P}(R_T)}
\newcommand{\SRT}{\mathscr{S}(R_T)}
\newcommand{\bRT}{\bar R_T}
\newcommand{\POmega}{\mathscr{P}\Omega}

\DeclareMathOperator{\diam}{diam}
\DeclareMathOperator{\Erf}{Erf}
\DeclareMathOperator{\sgn}{sgn}

\newcommand{\qand}{\quad\text{and}\quad}

\title[Fundamental gap in spheres]{Fundamental Gap of Convex Domains \\ in the Spheres \\
{\tiny 	--- with Appendix\,B by Qi S. Zhang }}

\author{Chenxu He}
\address{Department of Mathematics, University of California, Riverside, CA 92521}
\email{chenxuhe@math.ucr.edu}
\urladdr{https://sites.google.com/site/hechenxu/home}

\author{Guofang Wei}
\address{Department of Mathematics, University of California, Santa Barbara, CA 93106}
\email{wei@math.ucsb.edu}
\urladdr{http://web.math.ucsb.edu/~wei/}
\thanks{G. Wei is partially supported by NSF DMS 1506393}

\keywords{parabolic equation, eigenvalue estimate, spectral gap}

\subjclass[2000]{35P15, 35K20, 58J35, 53C35}

\begin{document}
\maketitle

\begin{abstract}
In \cite{SetoWangWei}, S. Seto, L. Wang and G. Wei proved that the gap between the first two Dirichlet eigenvalues of a convex domain in the unit sphere is at least as large as that for an associated operator on an interval with the same diameter, provided that the domain has the diameter at most $\pi/2$. In this paper, we extend Seto-Wang-Wei's result to convex domains in the unit sphere with diameter less than $\pi$. 
\end{abstract}

\section{Introduction}

Let $(M^n, g)$ be a Riemannian manifold, and $\Omega \subset M$ a bounded convex domain with the diameter $D$. Consider the Laplace operator $\Delta$ of $\Omega$ with the Dirichlet boundary condition. It has an increasing sequence of eigenvalues $0 <\lambda_0 < \lambda_1 \leq \lambda_2 \leq \cdots$, and corresponding eigenfunctions $\set{\phi_i}_{i\geq 0}$ which vanish on $\pd \Omega$ and satisfy the equation
\begin{equation*}
\Delta \phi_i + \lambda_i \phi_i = 0.
\end{equation*}
 The difference between the first two eigenvalues, $\lambda_1 - \lambda_0$, is called the \emph{fundamental gap}. Similarly, the fundamental gap is also defined for the Schr\"odinger operator $-\Delta + V$, where $V$ is a potential.
 
 When the Riemannian manifold $M^n$ is the Euclidean space $\Real^n$ with the flat metric, in their celebrated work \cite{AndrewsClutterbuck}, B. Andrews and J. Clutterbuck proved  the sharp lower bound $\lambda_1 - \lambda_0 \geq \frac{3\pi^2}{D^2}$, and thus resolved the fundamental gap conjecture which was independently proposed by M. van den Berg \cite{vandenBerg}, Ashbaugh-Benguria \cite{AshbaughBenguria} and Yau \cite{Yau} in the 80's. They proved this optimal lower bound by establishing a sharp log-concavity estimate for the first eigenfunction. Later, L. Ni \cite{Ni} gave an alternative proof of Andrews-Clutterbuck's results, using the maximum principle of elliptic equations. 

When the Riemannian manifold $M^n$ is the round sphere $\sph^n$, S. Seto, L. Wang and G. Wei have made substantial progress in their very recent preprint \cite{SetoWangWei}, by extending the work in \cite{AndrewsClutterbuck} to the spherical case. In particular, they obtained the lower bound $\frac{3\pi^2}{D^2}$(for $n \geq 3$) which is the same as the Euclidean case. From the geometric point of view, the analysis shouldn't change much before the domain reaches the equator of the sphere. So it is natural to expect that Seto-Wang-Wei's results should hold on all convex domains in the hemisphere. In this paper, we confirm this expectation. Below, $\sph^n(1)$ stands for the round sphere with radius one.

\begin{thm}\label{thm:gapcomparison}
Suppose $\Omega \subset \sph^n(1)$ is a strictly convex domain with the diameter $D< \pi$. Let $0< \lambda_0 < \lambda_1$ be the first two eigenvalues of the Laplacian on $\Omega$ with Dirichlet boundary condition. Then we have
\begin{equation}\label{eqn:gapcomparison}
\lambda_1 - \lambda_0 \geq \mu_1(n, D) - \mu_0(n, D)
\end{equation}
where $\mu_i(n, D)$($i=0,1$) are the first two eigenvalues of the Sturm-Liouville operator 
\begin{equation}\label{eqn:SLonedim}
\frac{d^2}{ds^2} - (n-1) \tan(s) \frac{d}{ds} \quad \text{for } s\in\sqset{- D/2, D/2}
\end{equation}
with Dirichlet boundary condition. Furthermore, if $n\geq 3$, then 
\[
\mu_1(n, D) - \mu_0(n, D) \geq 3 \frac{\pi^2}{D^2}. 
\]
\end{thm}  
An immediate consequence is the following 
\begin{cor}\label{cor:gap3pisq}
Suppose $\Omega \subset \sph^n(1)$($n \geq 3$) is a strictly convex domain with the diameter $D < \pi$. Let $0< \lambda_0 < \lambda_1$ be the first two eigenvalues of the Laplacian on $\Omega$ with Dirichlet boundary condition. Then we have
\begin{equation}\label{eqn:gap3pisq}
\lambda_1 - \lambda_0 \geq 3\frac{\pi^2}{D^2}. 
\end{equation}
\end{cor}

\begin{rem}
\begin{enumerate}[(a)]
\item The comparison of gaps in (\ref{eqn:gapcomparison}) and the inequality in (\ref{eqn:gap3pisq}) were proved for $D \leq \frac{\pi}{2}$ in \cite{SetoWangWei}.

\item The same estimates hold for Schr\"{o}dinger operator of the form $-\Delta + V$, where $V$ is a non-negative convex potential. 
\end{enumerate}
\end{rem}
\begin{rem}
Previously, some weaker lower bounds were obtained by gradient estimate method: Lee-Wang \cite{LeeWang} showed a lower bound $\frac{\pi^2}{D^2}$. In \cite{LingMMJ,LingCAG}, J. Ling improved Lee-Wang's result to $\lambda_1 - \lambda_0 >\frac{\pi^2}{D^2}$  by using Yu-Zhong's work \cite{YuZhong}. 
\end{rem}
We obtain Theorem \ref{thm:gapcomparison} by extending the log-concavity estimate of the first eigenfunction in \cite{SetoWangWei} to convex domains with diameter $D < \pi$. 

\begin{thm}\label{thm:logconcavity}
Suppose $\Omega \subset \sph^n(1)$ is a strictly convex domain with diameter $D < \pi$, and $\phi_0$ is a first eigenfunction of the Laplacian on $\Omega$ with Dirichlet boundary condition. Then for any $x, y \in \Omega$ with $x\ne y$, we have 
\begin{equation}\label{eqn:logconcavity}
\bset{\nabla \log \phi_0(y), \gamma'(d/2)} - \bset{\nabla \log \phi_0(x), \gamma'(-d/2)} \leq 2\pset{\log\tilde\phi_0}' \pset{\frac{d(x,y)}{2}} 
\end{equation}
where $d=d(x,y)$, $\gamma$ is the normal minimizing geodesic with $\gamma(-d/2) = x$ and $\gamma(d/2) = y$, and $\tilde\phi_0 > 0 $ is a first eigenfunction of the Sturm-Liouville operator in (\ref{eqn:SLonedim}) with Dirichlet boundary condition. 
\end{thm}
\begin{rem}
The log-concavity estimate of the first eigenfunction in (\ref{eqn:logconcavity}) was proved for convex domains with diameter $D\leq \frac{\pi}{2}$ in \cite{SetoWangWei}.
\end{rem}

The connection between the log-concavity estimate of the first eigenfunction and the lower bound of the fundamental gap originates from the landmark work in \cite{SingerWongYauYau}, where they derived the lower bound $\frac{\pi^2}{4D^2}$ in the Euclidean case, from the fact that the first eigenfunction is log-concave. The improved relation has been established in  \cite[Proposition 3.2]{AndrewsClutterbuck} in the Euclidean case, and \cite[Theorem 4.1]{SetoWangWei} for Riemannian manifold with a lower Ricci curvature bound. More precisely, they proved that the estimate in (\ref{eqn:logconcavity}) implies the comparison in (\ref{eqn:gapcomparison}) for bounded convex domains. So Theorem \ref{thm:gapcomparison} follows directly from Theorem \ref{thm:logconcavity} by applying Theorem 4.1 in \cite{SetoWangWei}.  

One of the key ingredients in our proof of Theorem \ref{thm:logconcavity} is the preservation of initial modulus in \cite[Theorem 3.2]{SetoWangWei}, see Theorem \ref{thm:SWWpsiPDE}. To show the log-concavity estimate in (\ref{eqn:logconcavity}), we follow the original strategy by Andrews-Clutterbuck in \cite{AndrewsClutterbuck} in the Euclidean case: We construct a rough initial data of modulus of concavity for $\log \phi_0$. Then the semi-linear parabolic equation in \cite[Theorem 3.2]{SetoWangWei} improves it to the one with $\tilde\phi_0$ in (\ref{eqn:logconcavity}). In \cite{SetoWangWei} the log-concavity estimate is proved using their Theorem 3.8, the elliptic version of their Theorem 3.2. While the proofs of Theorems 3.2 and 3.8 in \cite{SetoWangWei} are almost the same, the elliptic version gives log-concavity estimate directly, but needs additional assumption on the modulus function which leads to the restriction of diameter $\le \frac{\pi}{2}$. For the parabolic version, one doesn't have the additional assumption, but needs an initial modulus function and dedicated study of the parabolic equation in \cite[Theorem 3.2]{SetoWangWei}. 

Since the sphere $\sph^n$ has positive curvature, unlike the Euclidean space, the non-trivial curvature terms come into the parabolic equation. So the equation that we study in this paper is more involved than the one in Andrews-Clutterbuck's work \cite{AndrewsClutterbuck}. For example, our equation (\ref{eqn:uPDE}) contains the nonlinear term $u^2$ for the unknown function $u$, and the standard existence results, e.g., Chapter XII in \cite{Lieberman}, do not apply, even for a smooth initial-boundary data. In Appendix A, we prove the existence of classical solution when the initial-boundary data is smooth. Using an approximation argument, the parabolic equation (\ref{eqn:uPDE}) with continuous initial-boundary data has a classical solution in the interior. For the continuity of the solution at the initial time and to the spatial boundary, one has to estimate the solution's modulus of continuity. Similar estimates of other quasi-linear parabolic equations, e.g., graphical mean curvature flow and its anisotropic analogues, have been developed in \cite{AndrewsClutterbuckLipschitz,AndrewsClutterbuckTime}. 

We refer to the excellent survey \cite{Ashbaugh}(up to 2006) and the celebrated work of Andrews-Clutterbuck \cite{AndrewsClutterbuck}, for the importance of the fundamental gap and the long history of this problem. We also encourage interested reader to look at the survey \cite{Andrews} by Andrews for difficult but interesting problems in this area. 

\smallskip

The paper is organized as follows: In Section 2, we collect the preliminaries, and offer a brief outline of our proof of Theorem \ref{thm:logconcavity}. In Section 3, we construct the initial data of the modulus of log-concavity of the first eigenfunction $\phi_0$. In Section 4, we study the parabolic equation of the modulus of log-concavity with the initial data given in section 3. In Section 5, we prove Theorem \ref{thm:logconcavity}. Appendix A is for the existence of the solution to the parabolic equation, when the initial data is given by the approximating smooth function.  Appendix B is the proof of Lemma \ref{lem:psikt<0}, written by Qi S. Zhang. 

\medskip
\textbf{Acknowledgement}:  The authors would like to express gratitude to Qi S. Zhang for his help on parabolic partial differential equations, especially for the proof of Lemma \ref{lem:psikt<0} using the heat kernel. Later on, we found a proof using maximum principle. We also want to thank Ben Andrews for clarifying the proof of Corollary 4.4 in \cite{AndrewsClutterbuck}.   

\medskip

\section{Preliminaries and outline of proof of Theorem \ref{thm:logconcavity}}

In this section we set up the notations and recall a few relevant results from \cite{AndrewsClutterbuck} and \cite{SetoWangWei}. In the second part, we describe a brief outline of our proof of Theorem \ref{thm:logconcavity}. The details will be carried out in the later sections.

\subsection{Preliminaries}
A function $f$ is called \emph{semi-convex}, if $f(x) + c \, d(x_0, x)^2$ is convex for some $c$, where $d(x_0, \cdot)$ is the distance from a fixed reference point $x_0$.  
\begin{defn}
\begin{enumerate}[(a)]
\item Given a semi-convex function $f$ on a domain $\Omega$. A function $\psi : [0, \infty) \To \Real$ is called a  \emph{modulus of concavity} for $f$, if for every $x\ne y $ in $\Omega$, we have
\begin{equation}
\bset{\nabla f(y), \gamma'(d/2)} - \bset{\nabla f(x), \gamma'(-d/2)} \leq 2 \psi\pset{\frac{d(x,y)}{2}},
\end{equation}
where $\gamma$ is a normal minimizing geodesic with $\gamma(-d/2) = x$, $\gamma(d/2) = y$ and $d = d(x, y)$. 
\item The function $\psi$ is called a \emph{modulus of log-concavity} for a positive function $\phi$ on $\Omega$, if it is a modulus of concavity of $\log \phi$. 
\end{enumerate}
\end{defn}

In this paper, we also use ${\,}'$ and ${\,}{''}$ to denote the spatial derivatives, e.g., $\psi' = \frac{\pd}{\pd z}\psi$ and $\psi'' = \frac{\pd^2}{\pd z^2}\psi$ for a function $\psi = \psi(z, t)$. Recall Theorem 3.2 in \cite{SetoWangWei} for the unit round sphere, i.e., $K =1$, which plays an important role in proving Theorem~\ref{thm:logconcavity}. 
\begin{thm}[Theorem 3.2, \cite{SetoWangWei}]\label{thm:SWWpsiPDE}
Suppose $\Omega \subset \sph^n(1)$ is a uniformly convex domain with diameter $D  < \pi$. Let $\phi_0$ be a positive eigenfunction of the Laplacian on $\Omega$ with Dirichlet boundary condition, associated to the first eigenvalue $\lambda_0 > 0$, and 
\[
\begin{array}{rl}
 u : &  \Omega \times [0, \infty) \To \Real \\
 & u(x, t) = e^{-\lambda_0 t} \phi_0(x).
\end{array}
\]
Suppose $\psi_0 : [0, D/2] \To \Real$ is a Lipschitz continuous modulus of concavity for $\log \phi_0$. Let the function 
\[
\psi \in C\pset{[0, D/2] \times [0, \infty)}\cap C^\infty\pset{[0, D/2] \times (0, \infty)}
\] 
be a solution of 
\begin{equation}\label{eqn:psipdelambda0}
\left\{
\begin{array}{rcl}
\medskip 
\dfrac{\pd}{\pd t}\psi(z, t) & \geq & \psi''(z,t) + 2\psi'(z, t) \psi(z,t) \\
\medskip
& & - \tan(z) \big{[}(n+1)\psi'(z,t) + 2\lambda_0 + 2\psi^2(z,t)\big{]} \\
\medskip
& &- (n-1)(1-\tan^2(z))\psi(z, t) \\
\psi(\cdot, 0) & = & \psi_0(\cdot) \\
\psi(0, t) & = & 0.
\end{array}
\right.
\end{equation}
Then $\psi(\cdot, t)$ is a modulus of concavity for $\log u(\cdot, t)$ for each $t \geq 0$. 
\end{thm}

\begin{rem}
The improved log-concavity of the first eigenfunction was first proved by Andrews-Clutterbuck for the Euclidean domain, see \cite[Theorem 4.1]{AndrewsClutterbuck}. The non-trivial curvature terms in the spherical case come in, as the second and third lines in the differential inequality (\ref{eqn:psipdelambda0}) of $\psi(z,t)$. 
\end{rem}

\begin{rem}
For the proof of Theorem 3.2 in \cite{SetoWangWei} (see also Theorem 4.1 in \cite{AndrewsClutterbuck}), one only has to assume that 
\[
\psi \in C\pset{[0, D/2]\times [0, \infty)} \cap C^\infty \pset{(0, D/2)\times (0, \infty)}
\]
and $\psi(z, t)$ is Lipschitz continuous in $z$-variable on $[0, D/2]$. 
\end{rem}

\begin{rem}
For any $t \geq 0$, we have $\log u(\cdot, t) = \log \phi_0 - \lambda_0 t$. So the modulus of concavity of $\log u (\cdot, t)$ is the same as the modulus of log-concavity of $\phi_0$.
\end{rem}

\begin{rem}\label{rem:stationarysolutionlambda0}
Note that we can rewrite the right hand side of (\ref{eqn:psipdelambda0}) as
\begin{eqnarray*}
& & \psi''(z,t) + 2\psi'(z, t) \psi(z, t) - \tan(z) \big[ (n+1)\psi'(z, t) + 2\lambda_0 + 2\psi^2(z,t)\big] \\
& & - (n-1)\pset{1-\tan^2(z)}\psi(z, t) \\
& = & \frac{\pd}{\pd z}\pset{\psi'(z,t) + \psi^2(z,t) - (n-1)\tan(z) \psi(z, t) + \lambda_0} \\
& & - 2\tan(z) \pset{\psi'(z,t) + \psi^2(z,t) - (n-1)\tan(z) \psi(z, t) + \lambda_0}.
\end{eqnarray*}
So the stationary solution to the equality case in (\ref{eqn:psipdelambda0}) is given by 
\begin{equation}\label{eqn:psilambda0c}
\psi'(z) + \psi^2(z) - (n-1)\tan(z) \psi(z) + \lambda_0 = \frac{c}{\cos^2(z)}
\end{equation}
for some constant $c$. Let $\psi(z) = \pset{\log\phi}'(z)$ and then we have
\begin{equation}\label{eqn:philambda0c}
\phi''(z) - (n-1)\tan(z)\phi'(z) + \lambda_0 \phi(z) = \frac{c}{\cos^2(z)}\phi(z).
\end{equation}
Let $\varphi(z) = \phi(z) \cos^{\frac{n-1}{2}}(z)$ and the equation is 
\begin{equation}
\varphi''(z) -\frac{1}{4}\pset{\frac{(n-1)(n-3)}{\cos^2(z)} - (n-1)^2 - 4\lambda_0} \varphi(z) - \frac{c}{\cos^2(z)}\varphi(z) = 0.
\end{equation}
\end{rem}

\subsection{Outline of the proof of Theorem \ref{thm:logconcavity}}

The proof uses Theorem \ref{thm:SWWpsiPDE}. Recall that $\mu_0$ and $\tilde\phi_0 > 0$ are the first eigenvalue and eigenfunction of the Sturm-Liouville operator in (\ref{eqn:SLonedim}) with Dirichlet boundary condition. Instead of solving the parabolic equation in (\ref{eqn:psipdelambda0}) with the eigenvalue $\lambda_0$ of $\Delta$ on $\Omega$, we consider the following parabolic equation:
\begin{equation}\label{eqn:psipdemu0}
\left\{
\begin{array}{rcl}
\medskip 
\dfrac{\pd}{\pd t}\psi(z, t) & = & \psi''(z,t) + 2\psi'(z, t) \psi(z,t) \\
\medskip
& & - \tan(z) \big{[}(n+1)\psi'(z,t) + 2\mu_0 + 2\psi^2(z,t)\big{]} \\
\medskip
& &- (n-1)(1-\tan^2(z))\psi(z, t)
\end{array}
\right.
\end{equation}
on $(0, D/2) \times (0, \infty)$. Since $0 < \mu_0\leq \lambda_0$(see Lemma 3.12 in \cite{SetoWangWei}) and $\tan(z) > 0$ for $z\in [0, D/2]$, the solution $\psi(z,t)$ then satisfies the differential inequality in (\ref{eqn:psipdelambda0}). Our proof of existence of solution to equation (\ref{eqn:psipdemu0}) is more involved, though the existence of solution to equation (19) in \cite{AndrewsClutterbuck} follows from Theorem 12.25 in \cite{Lieberman} directly. For the convenience of the reader, we list the steps that will be carried out in Section 3, 4 and 5.
\begin{enumerate}[(I)]
\item For each fixed integer $k > 0$, we construct a piecewise smooth function $\psi_{k,0}(z)$ on $[0, D/2]$ such that $\psi_{k,0}(0) = 0$ and $\psi_{k,0}(D/2) = -k$. From the construction, $\psi_{k,0}$ is a modulus of log-concavity of $\phi_0$. It is similar to the construction of the initial modulus in \cite[Corollary 4.4]{AndrewsClutterbuck}. See Section 3. 
\item We solve the parabolic equation (\ref{eqn:psipdemu0}) and find a solution
\[
\psi_k(z, t) \in C\pset{[0, D/2]\times [0, \infty)}\cap C^\infty\pset{(0, D/2)\times (0, \infty)}
\] 
with the initial data $\psi_{k}(z,0) =\psi_{k,0}(z)$ and the boundary condition $\psi_{k}(0, t) = 0$ and $\psi_{k}(D/2, t) = - k$ for $t\geq 0$. Moreover, $\psi_k(z,t)$ is Lipschitz continuous in $z$-variable on $[0, D/2]$. Theorem \ref{thm:SWWpsiPDE} then implies that each $\psi_{k}(\cdot, t)$ is a modulus of log-concavity for $\phi_0$ for all $t \geq 0$. See Section 4.
\item We show that the solution $\psi_k(\cdot, t)$ obtained in the previous step converges uniformly to $\pset{\log\tilde\phi_0}'$ as $k \To \infty$ and $t\To \infty$. So the uniform limit is also a modulus of log-concavity of $\phi_0$. Hence Theorem \ref{thm:logconcavity} follows. See Section 5. 
\end{enumerate}

\medskip

\section{The initial modulus of log-concavity}

In this section, we construct a function $\psi_{k,0}$ on $[0, D/2]$ as an initial modulus of log-concavity for $\phi_0$, see Proposition \ref{prop:psik0initialdata}. It is similar to the one in the Euclidean case, see pp. 912--913 in \cite{AndrewsClutterbuck}. 

Fix $D\in (0, \pi)$. Let $\tilde\phi_0$ and $\mu_0$ be the first eigenfunction and eigenvalue of the one dimensional Sturm-Liouville operator in (\ref{eqn:SLonedim}), i.e., 
\begin{equation}\label{eqn:ODEphi0mu0}
\tilde \phi_0''(z) -(n-1)\tan(z) \tilde \phi_0'(z) = -\mu_0 \tilde \phi_0
\end{equation}
with $\tilde \phi_0(D/2) = 0$ and $\tilde \phi_0'(0) = 0$. We also normalize $\tilde \phi_0'(D/2) = -1$ for convenience. First we solve a Robin eigenvalue problem.
\begin{lem}\label{lem:RobinEigenfunction}
For any $\eps > 0$, there exists a Robin eigenfunction $\tilde \phi_{0,\eps}$ with $c(\eps)> 0$, such that 
\[
\pset{\tilde\phi_{0,\eps}}'' - (n-1)\tan(z)\pset{\tilde\phi_{0,\eps}}' + \mu_0 \tilde\phi_{0,\eps} = \frac{c(\eps)}{\cos^2(z)}\tilde\phi_{0,\eps}
\]
on $[0, D/2]$, and 
\begin{eqnarray*}
\tilde \phi_{0,\eps} (D/2) = \eps; \quad \tilde \phi_{0,\eps}'(D/2) = -1; \\
\tilde \phi_{0, \eps}'(0) = 0; \quad \tilde \phi_{0,\eps} > 0 \text{ on } [0, D/2].
\end{eqnarray*}
\end{lem}
\begin{proof}
The argument is similar to the one in \cite[pp. 905-906]{AndrewsClutterbuck}. Denote $m = (n-1)/2$. Let $\varphi_0(z) = \tilde\phi_0(z) \cos^{m}(z)$ and 
\[
\tilde V(z) = \frac{1}{4}\pset{\frac{(n-1)(n-3)}{\cos^2(z)} - (n-1)^2 - 4\mu_0}.
\]
Then we have $\varphi_0''(z) = \tilde V(z) \varphi_0(z)$ with 
\[
\varphi_0'(0) = 0; \quad \varphi_0(D/2) = 0; \qand \varphi_0(z) > 0 \text{ for } z \in [0, D/2).
\]
Recall the Pr\"ufer transformation: if 
\[
\varphi'' = \pset{\tilde V + \frac{c}{\cos^2(z)}}\varphi,
\]
then $q = \arctan(\varphi' /\varphi)$ satisfies the first order ODE:
\begin{equation}\label{eqn:ODEq}
\frac{d}{dz} q - \pset{\tilde V + \frac{c}{\cos^2(z)}} \cos^2 q + \sin^2 q = 0.
\end{equation}
Let $q(z, q_0, c)$ be the solution of ODE (\ref{eqn:ODEq}) with $q(0, q_0, c) = q_0$. The ODE comparison implies that $q$ is strictly increasing in $q_0$ for all $z$, and also strictly increasing in $c$ for $z > 0$. The choice of $c=0$ and $q_0= 0$ corresponds to $\varphi_0$, and so $q(D/2, 0, 0) = - \pi/2$ and $q(z, 0, 0)\in (-\pi/2, \pi/2)$ for $0< z< D/2$.

Since $q(D/2, 0, c)$ is strictly increasing in $c$ and $q(D/2, 0, 0) = -\pi/2$, for $\eps > 0$, let
\[
\sigma = \frac{\eps}{1+ \eps m \tan(D/2)},
\]
and then there exists a unique $c(\eps) > 0$ such that $q(D/2, 0, c(\eps)) = \arctan \sigma - \pi/2$ and $q(z, 0, c(\eps)) \in (-\pi/2, \pi/2)$ for $0<z<D/2$. The corresponding $\varphi$ is given by 
\[
\varphi_{0,\eps}(z) = \varphi_{0,\eps}(D/2) \exp \pset{-\int_z^{D/2} \tan q(s,0,c(\eps))ds}
\]
and then
\begin{eqnarray*}
\tilde\phi_{0,\eps}(z) = \eps \frac{\cos^m(D/2)}{\cos^m(z)} \exp \pset{- \int_z^{D/2} \tan q(s, 0, c(\eps))ds},
\end{eqnarray*}
where we used the value $\tilde \phi_{0,\eps}(D/2) = \eps$. It follows that $\tilde \phi_{0,\eps}(z) >0$ for $z\in [0, D/2]$. The derivative is
\[
\tilde\phi_{0,\eps}'(z) = \big[m\tan(z) + \tan q(z, 0, c(\eps))\big]\tilde\phi_{0,\eps}(z).
\]
So we have $\tilde\phi'_{0,\eps}(0) = 0$ and 
\[
\tilde\phi'_{0,\eps}(D/2) = \pset{m \tan(D/2) - \frac{1}{\sigma}}\eps = -1.
\]
This finishes the proof.
\end{proof}

Next we construct the initial data $\psi_{k,0}$. The stationary solution $\psi(z)$ to the parabolic equation (\ref{eqn:psipdemu0}) solves the following ODE:
\begin{equation}\label{eqn:ODEpsimu0}
\psi'(z) + \psi^2(z) -(n-1)\tan(z) \psi(z)+\mu_0 = \frac{c}{\cos^2(z)}
\end{equation}
for some constant $c$ (see Remark \ref{rem:stationarysolutionlambda0}). Fix an integer $k> 0$ and consider the solutions $\psi^L_c$ and $\psi^R_{k, c}$ with 
\[
\psi^L_{c}(0) = 0 \qand \psi^R_{k, c}(D/2) = - k. 
\]
Then from the ODE comparison to $\varphi(z) = \phi(z)\cos^{\frac{n-1}{2}}(z)$, we have
\begin{enumerate}[(a)]
\item The solution $\psi^L_{c}(z)$ is strictly increasing in $c$ for $0 < z \leq D/2$.
\item For fixed $k> 0$, the solution $\psi^R_{k, c}(z)$ is strictly decreasing in $c$ for $0\leq z < D/2$.
\end{enumerate}

Let $\eps = 1/k$ and $\tilde\phi_{0, 1/k}(z)$($0\leq z \leq D/2$) be the Robin eigenfunction with $c= c(1/k)>0$ in Lemma \ref{lem:RobinEigenfunction}. Let
\begin{equation}
\psi^L_{c(1/k)} = \psi^R_{k, c(1/k)} = \tilde{\psi}_{k,0}
\end{equation}
with
\begin{equation}\label{eqn:tildepsik0} 
\tilde\psi_{k,0} = \pset{\log \tilde\phi_{0,1/k}}'.
\end{equation}
\begin{prop}\label{prop:uniquenesstpsi}
For each integer $k > 0$, $\tilde \psi_{k,0}$ in (\ref{eqn:tildepsik0}) is the unique stationary solution in $C^2\pset{[0, D/2]}$ to the parabolic equation (\ref{eqn:psipdemu0}) with boundary condition $\psi(0) = 0$ and $\psi(D/2) = -k$.
\end{prop}
\begin{proof}
It follows from the monotonicity of $\psi_c^L$ in $c$. Consider the boundary condition at $z = 0$ and then the stationary solution is given by $\psi = \psi_c^L$ for some constant $c$. The boundary condition $\psi(D/2) = -k$ determines $c$ as $c = c(1/k)$. So we have $\psi = \tilde\psi_{k,0}$. This finishes the proof.
\end{proof}

In the following, we derive the upper bounds of $\psi^L_{c}$ and $\psi^R_{k,c}$. Let
\[
p(z) = \psi(z) - \frac{n-1}{2}\tan(z),
\]
and then we have
\begin{equation*}
p'(z) = \psi'(z) - \frac{n-1}{2}\sec^2(z).
\end{equation*}
It follows that $p(z)$ solves the ODE:
\begin{equation}\label{eqn:pVc}
p'(z) + p^2(z) =  V(z)
\end{equation}
with
\begin{equation}\label{eqn:Vcmu0}
V(z) = \frac{(n-1)(n-3)+ 4c}{4\cos^2(z)} - \frac{(n-1)^2}{4} - \mu_0.
\end{equation}
The boundary values of $p(z)$ are given by 
\[
p(0) = \psi(0)\qand p(D/2) = \psi(D/2) - \frac{n-1}{2}\tan(D/2). 
\]
Denote $p^L_{c}(z)$ the solution with $p(0) = 0$, and $p^R_{\tilde k, c}(z)$ the solution with $p(D/2) = -\tilde k$ with
\begin{equation}\label{eqn:tildek}
\tilde k = k + \frac{n-1}{2}\tan(D/2).
\end{equation}
The function $V(z)$ is monotone, and the infimum and supremum are:
\begin{enumerate}[(a)] 
\item if $(n-1)(n-3)+ 4c \geq 0$, then 
\begin{eqnarray*}
\inf V & = & c - \frac{n-1}{2} - \mu_0 \\ 
\sup V & = & \frac{4c+(n-1)(n-3)}{4\cos^2(D/2)} - \frac{(n-1)^2}{4} - \mu_0;
\end{eqnarray*}
\item if $(n-1)(n-3)+4c < 0$, then 
\begin{eqnarray*}
\inf V & = & \frac{4c+(n-1)(n-3)}{4\cos^2(D/2)} - \frac{(n-1)^2}{4} - \mu_0 \\
\sup V & = & c - \frac{n-1}{2} - \mu_0.
\end{eqnarray*}
\end{enumerate}
In either case, let $\lambda_{\pm}\geq 0$ be real numbers such that $\lambda_+^2 \geq \sup V$ and $\lambda_-^2 \geq - \inf V$. Riccati equation comparison gives the upper bounds of $p^L_c$ and $p^R_{\tilde k, c}$: 
\begin{equation*}
p^L_{c} (z) \leq \lambda_+ \tanh(\lambda_+ z), \quad \text{ for } z \in [0, D/2];
\end{equation*}
and
\begin{eqnarray*}
p^R_{k, c}(z) & \leq & \lambda_- \tan\pset{\lambda_- \pset{D/2 - z} - \arctan\pset{\frac{\tilde k}{\lambda_-}}} \nonumber \\
& = & \frac{\lambda_- \tan\pset{\lambda_-\pset{\frac{D}{2} - z}}-\tilde k}{1+ \frac{\tilde k}{\lambda_-}\tan\pset{\lambda_-\pset{\frac{D}{2}- z}}}, \quad \text{ if } z > \frac{D}{2} - \frac{\frac{\pi}{2}+\arctan\pset{\frac{\tilde k}{\lambda_-}}}{\lambda_-}
\end{eqnarray*}
where $\tilde k$ is given in (\ref{eqn:tildek}). Then we have 
\begin{equation}\label{eqn:psiLupperbound}
\psi^L_{c} (z) \leq \lambda_+ \tanh(\lambda_+ z) + \frac{n-1}{2}\tan(z), \quad \text{ for } z \in [0, D/2];
\end{equation}
and
\begin{eqnarray}\label{eqn:psiRupperbound}
\psi^R_{k, c}(z) & \leq & \lambda_- \tan\pset{\lambda_- \pset{D/2 - z} - \arctan\pset{\frac{\tilde k}{\lambda_-}}} + \frac{n-1}{2}\tan(z) \nonumber \\
& = & \frac{\lambda_- \tan\pset{\lambda_-\pset{\frac{D}{2} - z}}-\tilde k}{1+ \frac{\tilde k}{\lambda_-}\tan\pset{\lambda_-\pset{\frac{D}{2}- z}}} + \frac{n-1}{2}\tan(z), \\ 
& & \qquad \text{ if } z > \frac{D}{2} - \frac{\frac{\pi}{2}+\arctan\pset{\frac{\tilde k}{\lambda_-}}}{\lambda_-} \nonumber
\end{eqnarray}
where $\tilde k = k + \frac{n-1}{2} \tan(D/2)$.

For any $s \geq 0$, both $\psi^L_{c(1/k) +s}(z)$ and $\psi^R_{k, c(1/k)-s}(z)$ are bounded from below by $\psi^L_{c(1/k)}(z)$, and have upper bounds in (\ref{eqn:psiLupperbound}) and (\ref{eqn:psiRupperbound}). So the supersolution $\psi^+_{k,s}$ is given by 
\[
\psi^+_{k,s}(z) = \min\set{\psi^L_{c(1/k)+s}(z), \, \psi^R_{k, c(1/k)-s}(z)}.
\]
Note that for $s>0$, since 
\[
\psi^L_{c(1/k)+s}(D/2) > \psi^L_{c(1/k)}(D/2) = -k = \psi^R_{k,c(1/k)-s}(D/2)
\]
and
\[
\psi^R_{k,c(1/k)-s}(0) > \psi^R_{k, c(1/k)}(0) = 0 = \psi^L_{c(1/k)+s}(0),
\]
the functions $\psi^L_{c(1/k)+s}$ and $\psi^R_{k,c(1/k)-s}$ intersect in the interval $(0, D/2)$, and so $\psi^+_{k,s}$ is a piecewise-smooth function. 

Next we derive the lower bound of $\psi^+_{k,s}$ for large $s > 0$. Write $V_k(z) = V(z)$ in (\ref{eqn:Vcmu0}) when $c = c(1/k)$. For 
\[
s > \max\set{- \inf_{z\in [0, D/2]} V_k(z), \, \sup_{z\in [0, D/2]} V_k(z)}, 
\]
let 
\begin{equation}\label{eqn:tildelambda}
\tilde\lambda_+ = \sqrt{s + \inf_{z\in [0, D/2]} V_k(z)} \qand \tilde\lambda_- = \sqrt{s - \sup_{z\in [0, D/2]} V_k(z)}.
\end{equation}
Note that $p^L_{c(1/k)+s}$ solves the equation
\[
p'(z) + p^2(z) = V_k(z) + \frac{s}{\cos^2(z)}.
\]
It follows that 
\[
p'(z) + p^2(z) \geq \tilde\lambda_+^2
\]
and then
\[
p^L_{c(1/k)+s} (z) \geq \tilde \lambda_+ \tanh\pset{\tilde\lambda_+ z}
\]
for $z\in [0, D/2]$. Similarly, $p^R_{\tilde k, c(1/k)-s}$ solves the equation
\[
p'(z) + p^2(z) = V_k(z) - \frac{s}{\cos^2(z)}\leq - \tilde\lambda_-^2.
\]
It follows that, if $z > \dfrac{D}{2} - \dfrac{\frac{\pi}{2} + \arctan\pset{\tilde k/\tilde\lambda_-}}{\tilde\lambda_-}$, then we have
\begin{eqnarray*}
p^R_{\tilde k, c(1/k) - s} (z) & \geq & \tilde\lambda_- \tan\pset{\tilde\lambda_- \pset{D/2 - z} - \arctan\pset{\frac{\tilde k}{\tilde\lambda_-}}} \\
& = & \frac{\tilde \lambda_- \tan\pset{\tilde \lambda_- (D/2- z)}-\tilde k}{1+ \frac{\tilde k}{\tilde \lambda_-}\tan\pset{\tilde \lambda_-(D/2 - z)}}.
\end{eqnarray*}
These give us the lower bound of $\psi^+_{k,s}$ as 
\begin{equation}
\psi^+_{k,s}(z) \geq 
\left\{
\begin{array}{ll}
\medskip
\tilde \lambda_+ \tanh\pset{\tilde \lambda_+ z} + \dfrac{n-1}{2}\tan(z), & 0\leq z \leq z_0 \\
\dfrac{\tilde\lambda_- \tan\pset{\tilde\lambda_-\pset{D/2-z}}-\tilde k}{1+ \frac{\tilde k}{\tilde\lambda_- }\tan\pset{\tilde\lambda_- \pset{D/2 -z}}} + \dfrac{n-1}{2} \tan(z), & z_0 \leq z \leq \frac D 2;
\end{array}
\right.
\end{equation}
where $\tilde k = k + \frac{n-1}{2}\tan(D/2)$, and we take 
\[
z_0 > \frac{D}{2} - \frac{1}{\tilde\lambda_-}\pset{\frac{\pi}{2} + \arctan\pset{\frac{\tilde k}{\tilde \lambda_-}}}
\]
to make the two cases equal. 

\begin{lem}\label{lem:psiplusks}
Let $\phi_0$ be in Theorem \ref{thm:SWWpsiPDE}. For each $k>0$ large enough, there exists $s(k) \geq 0$ such that $\psi^+_{k,s}$ is a modulus of log-concavity for $\phi_0$. 
\end{lem}
\begin{proof}
It follows from the argument in \cite[Lemma 4.5]{AndrewsClutterbuck} as we throw away the positive term $\frac{n-1}{2}\tan(z)$ in the lower bound of $\psi^+_{k,s}$. Instead of Lemmas 4.2 and 4.3 in \cite{AndrewsClutterbuck} for the Euclidean case, we use Lemmas 3.4 and 3.5 in \cite{SetoWangWei} for the spherical case. 
\end{proof}

Define
\[
s(k) = \inf\set{s\geq 0: \, \psi^+_{k,s} \text{ is a modulus of log-concavity for }\phi_0}
\]
and then we choose the initial data $\psi_{k,0}$ as
\begin{equation}\label{eqn:psik0}
\psi_{k,0}(z) = \min\set{\psi^+_{j,s(j)}(z): \, 1\leq j \leq k}
\end{equation}
for $0\leq z \leq D/2$. Note that $\psi_{k,0}(z)$ is non-increasing in $k$ for $z\in (0, D/2)$. 
 
\begin{prop}\label{prop:psik0initialdata}
For each integer $k>0$, the function $\psi_{k,0}$ in (\ref{eqn:psik0}) satisfies the following properties.
\begin{enumerate}[(i)]
\item It is Lipschitz continuous on $[0, D/2]$.
\item It is a modulus of log-concavity of $\phi_0$ in Theorem \ref{thm:SWWpsiPDE} on $[0, D/2]$.
\item It satisfies the boundary condition $\psi_{k,0}(0)=0$ and $\psi_{k,0}(D/2) = -k$. 
\end{enumerate}
Furthermore, we have either
\begin{enumerate}[(a)]
\item $\psi_{k,0}=\tilde\psi_{k,0}$ on $[0, D/2]$, or
\item $\psi_{k,0}(z) > \tilde\psi_{k,0}(z)$ for any $z \in (0, D/2)$, where $\tilde\psi_{k,0}$ is given in (\ref{eqn:tildepsik0}).
\end{enumerate}
\end{prop}
\begin{proof}
Property (iii) is obviously, and property (ii) follows from Lemma \ref{lem:psiplusks}. For property (i), note that $\psi_{k,0}$ is defined as the minimal of functions $\psi$'s which satisfy the first order ODE (\ref{eqn:ODEpsimu0}). So the derivative of $\psi_{k,0}$ are bounded on $[0, D/2]$ and thus $\psi_{k,0}$ is Lipschitz continuous.

For the rest of the proof, we assume that $\psi_{k,0}$ is different from the function $\tilde\psi_{k,0}$. Note that for any $j\leq k-1$, we have $\tilde\psi_{j,0}(0) = \tilde\psi_{k,0}(0) = 0$ and 
\[
\tilde\psi_{j,0}(D/2) = -j > -k = \tilde\psi_{k,0}(D/2).
\]
From the ODE comparison of (\ref{eqn:ODEpsimu0}), it follow that 
\[
\tilde\psi_{k,0}(z) < \tilde\psi_{j,0}(z) \leq \psi_{j,s(j)}^+ (z)
\] 
for any $z\in (0, D/2]$. So we have $\psi_{k,0}\geq \tilde\psi_{k,0}$ on $[0, D/2]$. If $\psi_{k,0}(z_0) = \tilde\psi_{k,0}(z_0)$ for some $z_0\in (0, D/2)$, then $\psi_{k,0}(z_0) = \psi^+_{k,s(k)}(z_0)$. So we have $s(k) = 0$ and then $\psi_{k,0} = \tilde\psi_{k,0}$ on $[0, D/2]$. This finishes the proof. 
\end{proof}

\medskip

\section{The parabolic equation with initial-boundary condition $\psi_{k,0}$}

In this section and Section 5, we assume that $\psi_{k,0}$ is different from $\tilde\psi_{k,0}$. It follows from Proposition \ref{prop:psik0initialdata} that $\psi_{k,0}(z) > \tilde\psi_{k,0}(z)$ for all $z\in (0, D/2)$.  In this section, we show the existence of solution $\psi_k(z,t)$ with the initial data $\psi_{k,0}$ constructed in Section 3. See Theorem \ref{thm:uPDEcont} and Corollary \ref{cor:psikPDEcont}.

Consider the following initial-boundary value problem: 
\begin{equation}\label{eqn:psikPDE}
\left\{
\begin{array}{rcl}
\medskip
\dfrac{\pd}{\pd t}\psi_k & = & \psi_k'' + 2\psi_k' \psi_k -(n+1)\tan(z) \psi_k'  \\
\medskip
& & -2\tan(z) \psi_k^2 - (n-1)\pset{1-\tan^2(z)} \psi_k - 2\mu_0 \tan(z) \\
\medskip
& & \text{ on }(0, D/2) \times (0, \infty); \\
\medskip
\text{with}& &\psi_k(\cdot, 0) = \psi_{k,0}(\cdot), \\
& & \psi_k(0, t) = 0 \qand \psi_k(D/2, t) = -k.
\end{array}
\right.
\end{equation}
Here $\psi_{k,0}$ is the piecewise smooth function in (\ref{eqn:psik0}), and $\, '$, $\, ''$ stand for the spatial derivatives. Let 
\begin{equation}\label{eqnukpsik}
u(z,t) = \psi_k(z, t) - \tilde\psi_{k,0} \qand u_{0}(z) = \psi_{k,0}(z) - \tilde\psi_{k,0}(z).
\end{equation}
Then (\ref{eqn:psikPDE}) in $\psi_k$ is equivalent to the following one in $u$:
\begin{equation}\label{eqn:uPDE}
\left\{
\begin{array}{rcl}
\dfrac{\pd}{\pd t}u & = & u'' + 2 u u' + \sqset{2\tilde\psi_{k,0}(z)-(n+1)\tan(z)} u' - 2\tan(z) u^2 \\
& & + \sqset{2\tilde\psi_{k,0}'(z) - 4 \tan(z) \tilde\psi_{k,0}(z) - (n-1)\pset{1-\tan^2(z)} } u \smallskip \\
& & \text{ on }(0, D/2) \times (0, \infty); \medskip \\
\text{with}& & u (\cdot, 0) = u_{0}(\cdot),  \\
& & u(0, t) =  u(D/2, t) = 0.
\end{array}
\right.
\end{equation}
Let $T > 0$ be a positive time. Denote $R_T = (0, D/2)\times (0, T]$ the rectangle (including the top at $t=T$), its side
\[
\SRT= \pset{\set{z=0} \times (0, T)} \cup \pset{\set{z=D/2}\times (0, T)}
\]
 and parabolic boundary
\[
\PRT = \pset{\set{z = 0}\times [0, T]} \cup \pset{[0, D/2]\times \set{ t= 0}} \cup \pset{\set{z = D/2}\times [0, T]}.
\] 
A function $u \in C^{2,1}\pset{R_T}$ means that $u$, $u'$, $u''$ and $\pd_t u$ exist and continuous on $R_T$. A solution $u$ of the initial-boundary value problem (\ref{eqn:uPDE}) is called \emph{classical} if $u \in C^{2,1}(R_T)\cap C\pset{\bar R_T}$.

Consider the following operator 
\begin{equation}\label{eqn:Pu}
P u  = - u_t + a^{11}u_{zz} + a(Z, u, u_z)
\end{equation}
with $Z = (z, t)$, $a^{11} = 1$ and
\begin{equation}\label{eqn:aZqp}
a(Z, q, p) = 2 q p + a_1(z) p -2\tan(z) q^2  + a_2(z) q
\end{equation}
with
\begin{eqnarray}
a_1(z) & = & 2\tilde\psi_{k,0}(z) - (n+1)\tan(z) \label{eqn:a1} \\
a_2(z) & = & 2\tilde\psi'_{k,0}(z) - 4\tan(z) \tilde\psi_{k,0}(z) - (n-1)\pset{1-\tan^2(z)}, \label{eqn:a2}
\end{eqnarray}
for $(Z, q, p) \in (0, D/2)\times (0, T) \times \Real\times \Real^1$. First we prove a comparison principle for the semi-linear parabolic operator $P$ in (\ref{eqn:Pu}). 
\begin{lem}[Comparison principle]\label{lem:comparison}
Suppose that $u, v$ are functions in $C^{2,1}(R_T) \cap C\pset{\bar R_T}$ such that $P u \geq P v$ in $R_T$ and $u \leq v$ on $\PRT$. Assume that either $u_z$ or $v_z$ has an upper bound on $R_T$, then $u \leq v$ on $\bar R_T$.
\end{lem}
\begin{proof}
The argument is similar to Theorem 9.1 and Corollary 9.2 in \cite{Lieberman}. Assume that $u_z$ is bounded from above. Let $M$ be an upper bound of $\abs{u}$, $\abs{v}$ and $u_z$ on $R_T$, and set $w = (u -v)e^{\lambda t}$ for some constant $\lambda$ to be determined. We have $w \leq 0$ on $\PRT$. Assume that $w$ has the positive maximum at $Z_0 = (z_0, t_0) \in R_T$. Then we have
\[
u_z(Z_0) = v_z(Z_0), \quad u_{zz}(Z_0) - v_{zz}(Z_0) \leq 0 \qand (u-v)_t(Z_0) + \lambda (u-v)(Z_0) \geq 0 
\]
Let $m = u(Z_0) - v(Z_0) > 0$. It follows that 
\begin{eqnarray*}
0 & \leq & Pu(Z_0) - Pv(Z_0) \\
& = & -(u -v)_t (Z_0) + (u - v)_{zz}(Z_0) + a\pset{Z_0, u(Z_0), u_z(Z_0)} - a\pset{Z_0, v(Z_0), v_z(Z_0)} \\
& \leq & \lambda m + 2 u_z(Z_0) m - 2\tan(z_0)(u+v)(Z_0) m + a_2(z_0) m \\
& = & m \pset{\lambda + 2u_z(Z_0) - 2\tan(z_0) (u+v)(Z_0) + a_2(z_0)}.
\end{eqnarray*}
This leads a contradiction if we take 
\[
\lambda < - 2M -4M \tan(D/2) - \sup_{z\in [0, D/2]}\abs{a_2(z)}.
\]
The proof when $v_z$ is bounded above is similar and so we finish the proof. 
\end{proof}

From the comparison principle, we deduce that any solution $u$ of (\ref{eqn:uPDE}) is between the supersolution $u = u_0$ and the subsolution $u = 0$. 
\begin{lem}\label{lem:ubounds}
If $u$ is a classical solution to the initial-boundary value problem (\ref{eqn:uPDE}) for some $T > 0$,  then we have
\begin{equation}
0 \leq u(z, t) \leq u_0(z)
\end{equation}
for all $(z, t) \in [0, D/2] \times [0, T]$, and equivalently, 
\begin{equation}
\tilde\psi_{k,0}(z) \leq \psi_{k}(z, t) \leq \psi_{k,0}(z).
\end{equation}
\end{lem}
\begin{proof}
The lower bound of $u$ follows directly from Lemma \ref{lem:comparison}. Next we consider the upper bound. 

From the construction in Section 3, $\psi_{k,0}(z)$ is given by either $\psi^L_c(z)$ or $\psi^R_{j, c}(z)$ for certain value of $c$ and $j \leq k$. These $\psi^L$ and $\psi^R$'s are bounded below by $\tilde\psi_{k,0}$. From the upper bound (\ref{eqn:psiLupperbound}), each $\psi^L_c$ is smooth on $[0, D/2]$. From the upper bound (\ref{eqn:psiRupperbound}), if $\psi^R$ is not smooth on $[0, D/2]$, then it is defined on $(z_0, D/2]$ for some $z_0 \in (0, D/2)$, and blows up to $\infty$ as $z\To z_0$. Since $u_0 = \psi_{k,0} - \tilde\psi_{k,0}$, it follows that there is a finite collection of smooth functions $\set{f_\alpha}_{\alpha=1}^{N_1}$ on $[0, D/2]$, and $\set{g_\beta}_{\beta=1}^{N_2}$ on $(z_\beta, D/2]$ such that
\begin{enumerate}
\item all $f_\alpha$ and $g_\beta$ are stationary solutions to the equation (\ref{eqn:uPDE}), and
\item we have 
\[
u_{0}(z) = \min\set{f_1(z), \cdots, f_{N_1}(z), g_1(z), \cdots, g_{N_2}(z)}, \text{ for all } z\in [0, D/2].
\]
\end{enumerate}
Here we assume $g_\beta(z) = \infty$ if $z\leq z_\beta$. 

Since $P u = 0$ and $P f_\alpha = 0$ in $R_T$, $u_0 \leq f_\alpha$ on $\PRT$, and $f_\alpha'(z)$ is bounded from above on $[0, D/2]$, it follows from Lemma \ref{lem:comparison} that $u(z, t) \leq f_\alpha(z)$ on $\bar R_T$. Next we show that $u(z,t)\leq g_\beta(z)$ for any $(z, t)\in (z_\beta, D/2]\times [0, T]$. Note that $u$ is continuous on $R_T$, and $g_\beta$ approaches to $\infty$ as $z\To z_\beta$. So there is $\tilde z > z_\beta$ such that $u(\tilde z, t) < g_\beta(\tilde z)$ for all $t \in [0, T]$. It follows that $u(z, t) \leq g_\beta(z)$ on the parabolic boundary of $[\tilde z, D/2]\times [0, T]$. Note that $g'_\beta(z)$ is bounded from above on $[\tilde z, D/2]$. It follows from Lemma \ref{lem:comparison} again, that $u(z,t) \leq g_\beta(z)$ on $[\tilde z, D/2]\times [0, T]$ and then $(z_\beta, D/2]\times [0, T]$. In summary, for any $t\in [0, T]$, the solution $u(z, t) \leq f_\alpha$ and $g_\beta$ on their domains, and so $u(z, t)\leq u_{0}(z)$ on $\bar R_T$. So we finish the proof. 
\end{proof}

We solve the initial-boundary value problem (\ref{eqn:uPDE}) in the following steps:
\begin{enumerate}[(i)]
\item Choose a number $\alpha \in (0, 1)$. Fix $\eps > 0$ and choose $u_0^\eps \in C^{2+\alpha}([0, D/2])$ such that 
\begin{eqnarray}\label{eqn:u0eps}
& 0\leq u_0^\eps \leq u_0; \quad \sup_{[0, D/2]} \abs{u_0^\eps - u_0}\leq \eps; \text{ and } & \\
&  u_0^\eps(z) = u_0(z) \, \text{ for all } z\in [0, \kappa]\cup [D/2 - \kappa, D/2] & \nonumber
\end{eqnarray}
where $\kappa > 0$ is a constant independent of $\eps$. Let $L_0 >0$ be the Lipschitz constant of $u_0$ on $[0, D/2]$. We also assume that 
\begin{equation}\label{eqn:u0epsLip}
\sup_{z \in [0, D/2]} \abs{\pset{u^\eps_0}'(z)} \leq 2L_0 \quad \text{for any } \eps.  
\end{equation}  
\item Solve problem (\ref{eqn:uPDE}) for the initial condition $u^\eps_0 \in C^{2+\alpha}$. 
\item Solve for $u_0$ using the approximating solutions in the previous step.
\end{enumerate}
We leave the details of step (ii) above in Appendix A, see Theorem \ref{thm:uPDEHolder}, since the existence of such solution $u^\eps\in C^{2,1}\pset{R_T}\cap C\pset{\bar R_T}$ follows from the standard theory of quasi-linear parabolic equation. Here we also collect the a prior estimates which will be used in step (iii). 

\begin{thm}\label{thm:existenceueps}
For small $\eps > 0$, suppose the initial data $u^\eps_0\in C^{2+\alpha}([0, D/2])$ satisfies the assumptions in (\ref{eqn:u0eps}) and (\ref{eqn:u0epsLip}). Then the initial-boundary value problem in (\ref{eqn:uPDE}) has a classical solution $u^\eps = u^\eps(z, t)$ with $u^\eps(\cdot, 0) = u_0^\eps(\cdot)$. Moreover, the solution $u^\eps$ has the following properties.
\begin{enumerate}[(i)]
\item (Maximum estimate) For any small $\eps > 0$, we have 
\[
\sup_{\bar R_T} \abs{u^\eps} \leq C_1 \quad \text{ with } C_1 = \sup_{z\in [0, D/2]} \abs{u_0(z)}.
\]
\item (Local gradient estimate) Let $\beta_0$, $\beta_1$ be positive constants such that 
\begin{equation*}
\abs{a(Z, q, p)} \leq \beta_0 p^2 \quad \text{for } \abs{p} \geq \beta_1 \text{ and } (Z, q)\in R_T \times [-C_1, C_1].
\end{equation*}
Then for small $\rho > 0$, we have
\begin{equation*}
\abs{u^\eps_z} \leq \exp \pset{8 \beta_0 C_1} \pset{2L_0 + \beta_1 + \frac{8 C_1}{\rho}}
\end{equation*}
on $(\rho, D/2 - \rho)\times (0, T)$, where $L_0$ is the Lipschitz constant of $u_0$ on $[0, D/2]$.
\end{enumerate} 
\end{thm}
\begin{proof}
The statement other than the one in (ii) is already in Theorem \ref{thm:uPDEHolder}. The local gradient estimate in (ii) follows from Theorem 11.17 in \cite{Lieberman} and the uniform Lipschitz constant of $u^\eps_0$ in assumption (\ref{eqn:u0epsLip}).
\end{proof}

From the standard argument of the existence for continuous initial data, see for example \cite{AndrewsClutterbuckTime}, we have 

\begin{thm}\label{thm:uPDEcont}
Consider the initial-boundary value problem in (\ref{eqn:uPDE}) with the initial data $u_0=\psi_{k,0} - \tilde\psi_{k,0}$.  Then there is a solution $u \in C^{\infty}\pset{(0, D/2)\times (0, T]} \cap C\pset{[0, D/2] \times [0, T]}$ for any $T > 0$. Moreover, the solution $u(z, t)$ is Lipschitz continuous in $z$ on the interval $[0, D/2]$, and the Lipschitz constant is independent of $t\in [0, \infty)$.
\end{thm}

\begin{proof}
For each small $\eps > 0$, there is a classical solution $u^\eps$ with the initial condition $u^\eps_0$. The maximum estimate in Theorem \ref{thm:existenceueps} implies $\sup \abs{u^\eps} \leq C_1$ on $\bRT$, where $C_1$ is independent of $\eps$. Let $U \subset \subset R_T$ be an interior set. The local gradient estimate in Theorem \ref{thm:existenceueps} implies a bound $C_2$ of $\abs{u_z^\eps}$ on $U$ which is independent of $\eps$. Theorem 12.2 of \cite{Lieberman} provides a H\"older estimate of $u_z^\eps$ on $U$, and then Theorem 4.9 of \cite{Lieberman} provides a $C^{2+\alpha}$ estimate of $u^\eps$ on $U$. Note that all these estimates are independent of $\eps$. It follows that a subsequence of $u^\eps$ converges in $C^{2+\beta}(U)$ for some positive $\beta$ less than $\alpha$, to a limit $u$ as $\eps \To 0$. Then the regularity theory, see e.g., Theorem 10 on p. 72 in \cite{Friedman}, of the linear equation $L u = 0$ with $L$ in (\ref{eqn:Lbc}) implies that $u \in C^\infty(U)$. 

For the solution $u \in C(\bRT)$, we have to show that $u$ is continuous upon the parabolic boundary $\PRT$. The continuity of $u$ on the sides $\SRT$ follows from Remark \ref{rem:boundarygradient}. In fact, we have 
\[
\abs{u(z, t) - u(\xi,s)} \leq \tilde C_2 \pset{\abs{z-\xi} + \abs{t-s}^{\frac 1 2}}
\]
for any $(z, t)\in \SRT$ and $(\xi,s)\in R_T$, where $\tilde C_2$ is given in Equation (\ref{eqn:boundarygradientcont}). In the following, we show that $u(\cdot, t)\To u_0$ uniformly as $t\To 0$. 

The barriers we will use are very close to those (for $n=1$) in \cite[Theorem 11]{AndrewsClutterbuckTime}. Let $\phi: [0, \infty) \To \Real$ be the unique solution of 
\[
\phi''(z) = \frac{1}{2}\pset{\phi(z) - z \phi'(z)}
\]
with $\phi'(0) = 0$ and $\lim_{z\To \infty}\phi'(z) = 1$. In fact we have
\[
\phi(z) = \frac{2}{\sqrt{\pi}}e^{-\frac{z^2}{4}} + z \Erf(z/2)
\]
where $\Erf$ is the Gauss error function. For each $\tau > 0$, let 
\begin{equation}
w(z, t) = u_0^\eps (z_0) + \tau + \mu_{\tau} t + 2L_0 \sqrt{t} \phi\pset{\frac{\abs{z - z_0}}{\sqrt t}}.
\end{equation} 
Note that $w \geq 0$ for $t\geq 0$, and $\sqrt{t} \phi\pset{\abs{z-z_0}/{\sqrt{t}}}$ is bounded near $t = 0$. 
Let 
\[
A_0 = \sup_{(z, t)\in [0, D/2]\times[0, 1]} \sqrt{t}\phi \pset{\abs{z-z_0}/\sqrt t},
\]
\[
A_1 = \sup_{z\in [0, D/2]}\abs{a_1(z)} \qand A_2 = \sup_{z\in [0, D/2]} \abs{a_2(z)}.
\]
Choose $\mu_{\tau}$ large, say 
\begin{equation}\label{eqn:muvtaulowerbound}
\mu_{\tau} > 2 L_0 A_1 + \pset{A_2 + 4L_0} \pset{C_1 + \tau +1 + 2L_0 A_0},
\end{equation}
where $C_1>0$ is in Theorem \ref{thm:existenceueps}. For $0\leq t \leq \mu_{\tau}^{-1}$, we have
\[
w \leq u_0^\eps(z_0) + \tau + 1 + 2L_0 A_0.
\]
We show that $w$ is a super solution.  First we have
\[
w' = \pd_z w = 2\sgn(z-z_0)L_0 \phi'\pset{\frac{\abs{z-z_0}}{\sqrt{t}}} \leq 2L_0,
\]
and then 
\begin{eqnarray*}
P(w)  & = & -\mu_{\tau} + a(z, t , w, w') \\
& = & - \mu_{\tau} + 2 w w' + a_1(z) w'  - 2\tan(z) w^2 + a_2(z) w \\
& = & -\mu_{\tau} + 2\sgn(z - z_0) L_0 a_1(z) \phi' \\
& & -2\tan(z) w^2 + \sqset{a_2(z) + 4\sgn(z -z_0)L_0 \phi'} w \\
& \leq & -\mu_{\tau} + 2L_0 A_1 -2\tan(z) w^2 + \pset{A_2 + 4 L_0} w \\
& \leq & - \mu_{\tau} + 2L_0 A_1 + \pset{A_2 + 4L_0} \pset{u_0^\eps(z_0) + \tau +1 + 2L_0 A_0}.
\end{eqnarray*}
It follows that $P(w) < 0$. When $t\To 0$, we have  
\[
w(z, 0) \To u_0^\eps(z_0) + \tau + 2L_0 \abs{z - z_0}\geq u_0^\eps(z).
\]
Note that $\pd_z w$ is bounded from above by $2L_0$. Then the comparison principle in Lemma \ref{lem:comparison} implies that $w(z, t) \geq u^\eps(z, t)$ for $0\leq t \leq \mu_{\tau}^{-1}$. In particular, we have
\[
u^\eps(z_0, t) - u^\eps_0(z_0) \leq \tau + \mu_{\tau} t + 2L_0 \phi(0) \sqrt{t}.
\]
Similarly, the lower bound can be derived from the following subsolution 
\[
w^{-}(z, t) = u_0^\eps (z_0) - \tau - \mu_{\tau} t -  2L_0 \sqrt{t} \phi\pset{\frac{\abs{z-z_0}}{\sqrt t}}.
\]
Now we show the uniform convergence in time for $u^\eps$. Given $\tau > 0$, let 
\[
\delta(\tau) = \min \set{\frac{\tau}{\mu_\tau}, \frac{\tau^2}{4L_0^2 \phi(0)^2}}.
\]
Then for $0\leq t \leq \delta(\tau)$ and any $x\in [0, D/2]$, 
\[
\abs{u^\eps(z,t) - u^\eps_0(z)} \leq 3\tau.
\]
The rest follows exactly in Theorem 11 of \cite{AndrewsClutterbuckTime}. 

To finish the proof we show that $u(z,t)$ is Lipschitz continuous in the first coordinate. Since $\abs{u^\eps_z}$ is bounded uniformly in $\eps$ on any interior set, we only have to show that there is a uniform bound of the gradient near the set $\SRT$. It follows from Remark \ref{rem:boundarygradient} and the assumption of $u_0^\eps$ in (\ref{eqn:u0eps}). Note that in (ii) of Theorem \ref{thm:existenceueps} and inequality (\ref{eqn:boundarygradientcont}) of Remark \ref{rem:boundarygradient}, the upper bounds do not depend on $t$. So the Lipschitz constant of $u(\cdot, t)$ is independent of $t$. 
\end{proof}

Theorem \ref{thm:uPDEcont} has the following 

\begin{cor}\label{cor:psikPDEcont}
The initial-boundary value problem (\ref{eqn:psikPDE}) has a classical solution  
\[
\psi_k(z, t) \in C^{\infty}\pset{(0, D/2)\times (0, \infty)}\cap C\pset{[0, D/2]\times [0, \infty)}.
\]
given by $\psi_{k}(z,t) = u(z,t) + \tilde\psi_{k,0}(z)$. 
Moreover, the gradient $\psi_k' (z, t)$ is bounded on $(0, D/2)$, $\psi_k(z,t)$ is Lipschitz continuous in $z$-variable on $[0, D/2]$ and the Lipschitz constant is independent of $t$. 
\end{cor}

\medskip

\section{Proof of Theorem \ref{thm:logconcavity}}

In this section, we prove Theorem \ref{thm:logconcavity} by showing the uniform convergence of the solutions $\psi_{k}(z,t)$, when $t \To \infty$ and $k \To \infty$. See the proof at the end of this section. 

First we show that the solution $\psi_{k}(z,t)$ of the initial-boundary value problem (\ref{eqn:psikPDE}) is monotone decreasing in $t$. Recall $u(z, t) = \psi_k(z, t) - \tilde\psi_{k,0}(z)$ in (\ref{eqnukpsik}). Let 
\begin{equation}\label{eqn:wpdtpsi}
w(z, t) =\pd_t \psi_{k}(z,t) = \pd_t u(z,t)
\end{equation}
and
\begin{eqnarray}
w_0(z)  & = & \psi_{k,0}''(z) + 2\psi'_{k,0}(z) \psi_{k,0}(z) - (n+1)\tan(z) \psi'_{k,0}(z) \nonumber \\
& & - 2\tan(z) \psi_{k,0}^2(z) - (n-1)\pset{1-\tan^2(z)}\psi_{k,0}(z) - 2\mu_0 \tan(z).\label{eqn:w0}
\end{eqnarray}
Note that $w_0 \in L^2\pset{[0, D/2]}$, $\psi'_{k,0}(z)$ and $\psi''_{k,0}(z)$ are derivatives in the weak sense. 

Differentiate equation in (\ref{eqn:psikPDE}) with respect to $t$, and then $w$ is a weak solution to the initial-boundary value problem of the following linear parabolic equation 
\begin{equation}\label{eqn:Lw}
\left\{
\begin{array}{rcl}
\medskip
L w  & = & -\pd_t w + w_{zz} + b(z) w_z + c(z) w = 0 \\
\medskip
& & \text{ on } (0, D/2) \times (0, \infty); \\
\medskip
\text{with} & & w(\cdot, 0) = w_0 \\
& & w(0, t) = w(D/2, t) = 0.
\end{array}
\right.
\end{equation}
where $w_0\in L^2\pset{[0, D/2]}$ is given in (\ref{eqn:w0}), and 
\begin{eqnarray*}
b(z) & = & 2\psi_k(z, t) - (n+1)\tan(z) \\
c(z) & = & 2\psi_k'(z,t) - 4\tan(z) \psi_k(z, t) - (n-1)\pset{1-\tan^2(z)}.
\end{eqnarray*}
From the comparison in Lemma \ref{lem:ubounds}, we have $w(z, 0) \leq 0$. This can also be verified directly for $w_0(z)$.  
\begin{lem}\label{lem:w0<=0}
We have $w_0(z) \leq 0$ in the weak sense. 
\end{lem}
\begin{proof}
Let $\eta(z) \geq 0$ be a smooth test function with compact support in $(0, D/2)$. Write
\begin{equation*}
\int_{0}^{D/2} w_0(z) \eta(z) dz = W_1(0, D/2) - W_2(0, D/2)
\end{equation*}
with 
\begin{eqnarray*}
W_1(\xi_1, \xi_2) & = & \int_{\xi_1}^{\xi_2} \sqset{\psi_{k,0}(z) \eta''(z) - \psi_{k,0}^2(z) \eta'(z) + (n+1)\psi_{k,0}(z) \pset{\tan(z) \eta(z)}' } dz \\
W_2(\xi_1, \xi_2) & = & \int_{\xi_1}^{\xi_2} \sqset{2\tan(z) \psi_{k,0}^2(z) + (n-1)\pset{1-\tan^2(z)}\psi_{k,0}(z)+ 2\mu_0 \tan(z)} \eta(z) dz.
\end{eqnarray*}
Let $[z_0, z_1] \subset [0, D/2]$ be a maximal interval where $\psi_{k,0}$ is smooth. Then integration-by-part yields 
\begin{eqnarray*}
W_1(z_0, z_1) & = & \int_{z_0}^{z_1} \set{\psi_{k,0}(z) \eta''(z) - \psi^2_{k,0}(z) \eta'(z) + (n+1)\psi_{k,0}(z) \pset{\tan(z) \eta(z)}' }dz \\
& = & \pset{\psi_{k,0}(z) \eta'(z) - \psi'_{k,0}(z) \eta(z) - \psi^2_{k,0}(z) \eta(z) + (n+1)\psi_{k,0}(z) \tan(z) \eta(z)} \Big{|}_{z_0}^{z_1} \\
& & +\int_{z_0}^{z_1}\pset{\psi''_{k,0}(z) + 2\psi'_{k,0}(z) \psi_{k,0}(z) - (n+1)\tan(z) \psi'_{k,0}(z)} \eta(z) dz.
\end{eqnarray*}
Note that $\psi_{k,0}(z)$ is a smooth stationary solution of equation (\ref{eqn:psikPDE}) in the interval $[z_0, z_1]$. So we have
\begin{eqnarray*}
&  & W_1(z_0, z_1) - W_2(z_0, z_1)  \\
& = & \pset{\psi_{k,0}(z) \eta'(z) - \psi'_{k,0}(z) \eta(z) - \psi^2_{k,0}(z) \eta(z) + (n+1)\psi_{k,0}(z) \tan(z) \eta(z)} \Big{|}_{z_0}^{z_1} 
\end{eqnarray*}
and then 
\begin{eqnarray*}
\int_0^{D/2} w_0(z) \eta(z) dz & = & W_1(0, D/2) - W_2(0, D/2) \\
& = & \sum_{i=1}^r \eta(z_i) \pset{\psi_{k,0}'(z_i +) - \psi_{k,0}'(z_i - )} \\
& \leq & 0
\end{eqnarray*}
where $\set{z_i}_{i=1}^r \subset (0, D/2)$ are points where $\psi'_{k,0}$ is discontinuous. This finishes the proof. 
\end{proof}
\begin{rem}
The statement that $w_0\leq 0$ in the weak sense implies the fact, that $\psi_{k,0}$ is a supersolution of the equation (\ref{eqn:psikPDE}).
\end{rem}
Note that $c(z)$ in equation (\ref{eqn:Lw}) is bounded. The weak maximum principle for weak solution of linear parabolic equation, see, for example, Corollary 6.26 in \cite{Lieberman} or Theorem III.7.2  in \cite{Ladyparabolic}, implies that $w(z, t) \leq 0$ on $[0, D/2]\times [0, \infty)$. Since $\psi_{k,0} \ne \tilde\psi_{k,0}$, the strong maximum principle, see Theorem 6.25 in \cite{Lieberman}, yields the following

\begin{lem}\label{lem:psikt<0}
For any $(z, t)\in (0, D/2) \times (0, \infty)$, we have
\begin{equation*}
w(z, t) = \frac{\pd}{\pd t} \psi_k(z, t) < 0. 
\end{equation*}
\end{lem}
\begin{rem}
See an alternative proof of Lemma \ref{lem:psikt<0} using the heat kernel in Appendix B by Qi Zhang. 
\end{rem}

\begin{prop}
For each $k > 0$, the solution $\psi_k(\cdot, t)$ converges to $\tilde\psi_{k,0}$ uniformly as $t\To \infty$.
\end{prop}
\begin{proof}
For fixed $t\geq 0$, from Corollary \ref{cor:psikPDEcont}, the solution $\psi_k(z, t)$ is Lipschitz in $z$-variable on $[0, D/2]$ and the Lipschitz constant is independent of $t$. It follows that $\set{\psi_k(\cdot, t)}_{t\geq 0}$ is equicontinuous. Lemma \ref{lem:ubounds} shows that it is also uniformly bounded. Arzel\`{a}-Ascoli Theorem then implies that there is a sequence $\set{t_n}$, such that $\psi_k(\cdot, t_n)$ uniformly converges to a continuous function, say $v$, on $[0, D/2]$ as $t_n \To \infty$. Since $\pd_t \psi_k(z,t) \leq 0$ for $t > 0$, $\psi_k(z, t)$ converges to $v(z)$ uniformly on $[0, D/2]$ as $t\To \infty$. In the following, we show that $v = \tilde\psi_{k,0}$. 

Since $\psi_k(z, t) \geq \tilde\psi_{k,0}(z)$, we have $v(z) \geq \tilde\psi_{k,0}(z)$ and is continuous on $[0, D/2]$. Consider the following non-homogenous linear parabolic equation 
\begin{equation}\label{eqn:wlinearnonhom}
- \pd_t Y + Y_{zz} + b(z, t) Y_z =f(z, t)
\end{equation}
on $(0, D/2)\times (0, \infty)$ with boundary condition $Y(0, t) =0$ and  $Y(D/2,t) = -k$, where
\begin{eqnarray*}
b(z, t) & = & 2\psi_{k}(z,t) - (n+1)\tan(z) \\
f(z, t) & = & 2\tan(z) \psi^2_{k}(z,t) + (n-1)\pset{1-\tan^2(z)} \psi_k(z,t) + 2\mu_0 \tan(z).
\end{eqnarray*}
 It follows that, as $t\To \infty$, 
\begin{eqnarray*}
b(z,t) & \To & 2 v(z) - (n+1)\tan(z) \\
f(z, t) & \To & 2\tan(z) v^2(z) + (n-1) \pset{1-\tan^2(z)}v(z) + 2\mu_0 \tan(z)
\end{eqnarray*}
uniformly on $[0, D/2]$, and the limits are continuous functions on $[0, D/2]$. Note that $f(z,t)$ is continuous on $[0, D/2] \times [0, \infty)$ and then Theorem 2 in \cite[p. 158]{Friedman} implies that, the solution of the parabolic equation (\ref{eqn:wlinearnonhom}) with the boundary conditions converges to the unique solution of the following ordinary differential equation
\begin{eqnarray}\label{eqn:yBVP}
y''(z) + \pset{2v(z) - (n+1)\tan(z)} y'(z) & = & 2\tan(z) v^2(z) \nonumber \\
& & + (n-1)\pset{1-\tan^2(z)}v(z) + 2\mu_0 \tan(z),
\end{eqnarray}
with the boundary condition $y(0) = 0$ and $y(D/2) = -k$. Since $\psi_k (z,t)$ solves the equation (\ref{eqn:wlinearnonhom}) and has the uniform limit $v(z)$, $v(z)$ solves the boundary value problem:
\[
\left\{
\begin{array}{rcl}
v''(z) + 2 v(z) v'(z) -(n+1)\tan(z) v'(z) - 2\tan(z) v^2(z) & & \medskip \\
 - (n-1)\pset{1-\tan^2(z)} v(z) - 2\mu_0 \tan(z) & = & 0, \medskip\\
 \text{ with } v(0) = 0 \text{ and } v(D/2) = -k. \qquad \qquad \qquad  & & 
\end{array}
\right.
\]
It follows that $v(z)$ is a stationary solution to the parabolic equation (\ref{eqn:psipdemu0}). So we conclude that $v= \tilde\psi_{k,0}$ from the uniqueness in Proposition \ref{prop:uniquenesstpsi}.  This finishes the proof.
\end{proof}

Finally we give the proof of Theorem \ref{thm:logconcavity} in Introduction.
\begin{proof}[Proof of Theorem \ref{thm:logconcavity}]
For each integer $k > 0$, from Proposition \ref{prop:psik0initialdata}, the initial data $\psi_{k,0}(z)$ is Lipschitz continuous, and a modulus of log-concavity of $\phi_0$ on $[0, D/2]$. $\psi_{k,0}$ also satisfies the boundary condition $\psi_{k,0}(0) = 0$ and $\psi_{k,0}(D/2) = -k$. From Corollary \ref{cor:psikPDEcont}, the initial-boundary value problem (\ref{eqn:psikPDE}) has the solution $\psi_k(z, t)$  in $C^\infty\pset{(0, D/2) \times (0, \infty)}\cap C\pset{[0, D/2]\times [0, \infty)}$, and it is Lipschitz continuous in $z$-variable on $[0, D/2]$. Since $\mu_0 \leq \lambda_0$, see Lemma 3.12 in \cite{SetoWangWei}, $\psi_{k}$ satisfies the parabolic differential inequality (\ref{eqn:psipdelambda0}). Theorem 3.2 in \cite{SetoWangWei}(or see Theorem \ref{thm:SWWpsiPDE}) implies that $\psi_k(\cdot, t)$ is a modulus of log-concavity of $\phi_0$ for all $t\geq 0$, so is the uniform limit $\tilde\psi_{k,0}$. Then we conclude that $\tilde\psi_0 = \pset{\log \tilde \phi_0}'$ is a modulus of log-concavity of $\phi_0$ as $\tilde\psi_0$ is the uniform limit of $\tilde\psi_{k,0}$ when $k \To \infty$. This finishes the proof.
\end{proof}

\medskip

\appendix

\section{Existence theorem for quasi-linear parabolic equations}

In this appendix, we give the detailed argument of step (ii) in solving the initial-boundary value problem (\ref{eqn:uPDE}). The main result is Theorem \ref{thm:uPDEHolder}. 

Fix $D < \pi$. In this appendix, we use $x, y, \ldots$ for the variable in $[0, D/2]$. Recall the standard notions used in the parabolic equations. For the interval $I = (0, D/2)\subset \Real^1$ and the cylindrical domain $\Omega = I \times (0, T)\subset \Real^2$, we define the bottom, corner and side of $\Omega$ as 
\[
B \Omega = I \times \set{0}, \quad C\Omega = \pd I \times \set{0}, \quad S\Omega = \pd I \times (0, T).
\]
The parabolic boundary is $\POmega =B\Omega \cup C \Omega \cup S\Omega$. Denote by $\Omega_t = I \times (0, t)$ the cylinder for any $t\in(0, T)$. For any two points $X = (x, t)$ and $Y = (y,s)$ in $\bar \Omega$, the parabolic distance is 
\[
\abs{X- Y} = \max \set{\abs{x-y}, \abs{s -t}^{1/2}}.
\]
The relevant H\"older norms for $0<\alpha \leq 1$:
\[
[u]_\alpha = \sup_{X\ne Y \in \Omega} \frac{\abs{u(X) - u(Y)}}{\abs{X -Y}^\alpha}
\] 
\[
\abs{u}_{1+\alpha} = \sup_\Omega\abs{u} + \sup_\Omega \abs{u_x} + [u_x]_\alpha
\]
and
\[
\abs{u}_{2+\alpha} = \sup_\Omega \abs{u} + \sup_\Omega\abs{u_x} + \sup_\Omega\abs{u_{xx}} + \sup_\Omega\abs{u_t} + [u_{xx}]_\alpha + [u_t]_\alpha.
\]
A function $u \in C^{2,1}\pset{\Omega}$ means that $u$, $u_x$, $u_{xx}$ and $u_t$ exist and are continuous on $\Omega$.

\smallskip

We drop the subscription $_{k, 0}$ in $\tilde\psi_{k,0}$ and write $\tilde\psi (x) = \tilde\psi_{k,0}(x)$. We also adopt the convention that $'$ and $''$ stand for the spatial derivatives. Recall the initial-boundary value problem in (\ref{eqn:uPDE}) for $u(x,t) = \psi_k(x,t) - \tilde\psi(x)$:
\begin{equation}\label{eqn:uPDEapp}
\left\{
\begin{array}{rcl}
\medskip
\dfrac{\pd}{\pd t}u & = & u_{xx} + 2 uu_x + \sqset{2\tilde\psi(x)-(n+1)\tan(x)} u_x - 2\tan(x) u^2 \\
\medskip
& & + \sqset{2\tilde\psi'(x) - 4 \tan(x) \tilde\psi(x) - (n-1)\pset{1-\tan^2(x)} } u \smallskip \\
& & \text{ on }[0, D/2] \times (0, \infty); \medskip \\
\medskip
\text{with}& & u(\cdot, 0) = u_0(\cdot),  \\
& & u(0, t) =  u(D/2, t) = 0,
\end{array}
\right.
\end{equation}
For the initial data $u_0$ in the problem (\ref{eqn:uPDEapp}) we assume that 
\begin{eqnarray}\label{eqn:u0assumption}
& 0 \leq u_0(x) \leq \psi_{k,0}(x) - \tilde \psi(x) \text{ for all }x \in [0, D/2], \text{ and} & \\
& u_0(x) = \psi_{k,0}(x) - \tilde\psi(x) \text{ for } x\in [0, \kappa] \cup [D/2 - \kappa, D/2].  & \nonumber 
\end{eqnarray}
Here $\kappa > 0$ is a constant determined by $\psi_{k,0} - \tilde \psi$. The initial data $u_0^\eps$ with assumption in (\ref{eqn:u0eps}) certainly satisfies the properties in (\ref{eqn:u0assumption}). 

\begin{thm}\label{thm:uPDEHolder}
Consider the initial-boundary problem in (\ref{eqn:uPDEapp}), and a constant $\alpha \in (0, 1)$. Suppose that $u_0$ is defined on the parabolic boundary $\POmega$ and $u_0 \in C^{2+\alpha}(\POmega)$ with the properties in (\ref{eqn:u0assumption}). Then there is a solution $u \in C^{2,1}(\Omega) \cap C(\bar \Omega)$
\end{thm}

A solution $u$ is said to be \emph{classical} if $u \in C^{2,1}(\Omega)\cap C(\bar \Omega)$. The proof follows a standard argument for showing existence: a bound on $\sup \abs{u}$, a bound on $\sup \abs{u_x}$, a H\"older gradient bound $\abs{u_x}_\gamma$ for some $\gamma \in (0, 1)$, and then the application of a fixed point theorem. 

Recall the parabolic operator in (\ref{eqn:Pu}) that is associated with the parabolic equation in (\ref{eqn:uPDEapp}):
\begin{equation}\label{eqn:Pux}
P = - u_t + a^{11}u_{xx} + a(X, u, u_x)
\end{equation}
with $a^{11} = 1$ and
\begin{equation}\label{eqn:aXqp}
a(X, q, p) = 2 q p + a_1(x) p -2\tan(x) q^2  + a_2(x) q
\end{equation}
with
\begin{eqnarray}
a_1(x) & = & 2\tilde\psi(x) - (n+1)\tan(x) \label{eqn:a1x} \\
a_2(x) & = & 2\tilde\psi'(x) - 4\tan(x) \tilde\psi(x) - (n-1)\pset{1-\tan^2(x)}, \label{eqn:a2x}
\end{eqnarray}
for $(X, q, p) \in (0, D/2)\times (0, T) \times \Real\times \Real^1$. 

Note that the trivial function is a sub-solution, and $\psi_{k,0} - \tilde\psi$ is a super-solution of the equation in (\ref{eqn:uPDEapp}). A similar argument as in Lemma \ref{lem:ubounds} yields 
\begin{lem}[Maximum estimate]\label{lem:Maxestimate}
If $u$ is a classical solution of (\ref{eqn:uPDEapp}) in $\Omega$, then 
\[
\sup_{\Omega} \abs{u} \leq C_1,
\]
where 
\[
C_1 =\sup_{[0, D/2]} \abs{\psi_{k,0}-\tilde\psi}
\] 
is independent of $u_0\in C^{2+\alpha}(\POmega)$.
\end{lem}

In the following, the positive constants $C_i$'s($i\geq 2$) may depend on $D$, $\sup\abs{\tan(x)}$, $\sup\abs{\tilde\psi(x)}$ and $\sup\abs{\tilde\psi'(x)}$. We only write out the explicit dependence of such $C_i$'s on $C_1$, $T$, $\sup \abs{u_0'}$, and $\sup \abs{u''_0}$. 

\begin{lem}[Boundary gradient estimate]\label{lem:boundarygradestimate}
If $u$ is a classical solution of (\ref{eqn:uPDEapp}) in $\Omega$, then
\begin{equation}
\sup_{\begin{array}{c}(x, t)\in S\Omega \\ (y,s)\in \Omega \end{array}} \frac{\abs{u(x,t) - u(y,s)}}{\abs{(x,t)- (y,s)}} \leq C_2,
\end{equation}
where $C_2$ is a constant depending on $C_1$, $\sup_{[0, D/2]}\abs{u_0'}$ and $\sup_{[0, D/2]}\abs{u_0''}$. 
\end{lem}
\begin{proof}
Define the auxiliary operator $\bar P$ as
\[
\bar P w = - w_t + w_{xx} + a(X, u(X), w_x)
\]
for any $w \in C^{2,1}(\Omega)$, where $a(X, q, p)$ is given in (\ref{eqn:aXqp}). Fix a point $X_0= (x_0, t_0) \in S\Omega$. First we assume that $x_0 = 0$. Choose a positive number $\sigma \leq D/2$ and the value of $\sigma$ is to be determined, and let $f = f(x)$ be an increasing $C^2$-function with $f(0) = 0$ to be determined. Assume that for any $x\in [0, \sigma]$, we have
\begin{equation}\label{eqn:fpassumption1}
f'(x) \geq \max\set{2\sup_{[0, D/2]}\abs{u_0'(x)}, 4}
\end{equation}
Let $w(x) = f(x) + u_0(x)$. Note that assumption (\ref{eqn:fpassumption1}) implies 
\begin{eqnarray}
& & \abs{w'(x)} = \abs{f'(x) + u_0'(x)} \leq 2f'(x) \label{eqn:wpupperbound} \\
& & \abs{w'(x)} = \abs{f'(x) + u_0'(x)} \geq \frac 1 2 f'(x) \geq 2. \label{eqn:wplowerbound}
\end{eqnarray}
For the term $a(X, u(X), p)$ we have the following estimate
\begin{eqnarray*}
a(X, u(X), p) & = & 2 u(X) p + a_1(x) p - 2\tan(x) u^2(X) + a_2(x) u(X) \\
& \leq & \alpha_1(C_1) \abs{p} + \alpha_2(C_1)
\end{eqnarray*}
where $\alpha_i$($i=1,2$) depends on $C_1$ in Lemma \ref{lem:Maxestimate}, and is independent of $u_0$. Let 
\[
\beta = \max\set{\alpha_1, \alpha_2}
\]
and then we have 
\begin{equation}\label{eqn:aupperbound}
a(X, u(X), p) \leq \beta p^2, \quad \text{if } \abs{p} \geq 2.
\end{equation}
Using inequality (\ref{eqn:wpupperbound}) and assumption $f'(x) \geq 4$ in (\ref{eqn:fpassumption1}), it follows that 
\begin{eqnarray*}
\bar P w & = & u_0''(x) + f''(x) + a\pset{X, u(X), w'(x)} \\
& \leq & u_0''(x) + f''(x) + \beta (w'(x))^2 \\
& \leq & f''(x) + 4\beta (f'(x))^2 + \frac{\abs{u_0''(x)}}{4} (f'(x))^2 \\
& < & f''(x) + \beta_0 (f'(x))^2
\end{eqnarray*}
for $\beta_0 = 4\beta + \sup \abs{u_0''}$. 

Set 
\[
M = \sup_\Omega \abs{u - u_0} \leq 2C_1.
\]
The ODE 
\[
f'' + \beta_0 (f')^2 = 0
\]
with $f(0) = 0$ and $f(\sigma) = M$ has the solution
\[
f(x) = \frac{1}{\beta_0}\log \pset{1+ \frac{\beta_0}{k_1}x} \quad \text{with}\quad \frac{\beta_0}{k_1}\sigma = e^{M \beta_0} -1.
\]
It follows that 
\[
f'(x) = \frac{1}{k_1 + \beta_0 x} \geq f'(\sigma) = \frac{1}{k_1+ \beta_0 \sigma} = \frac{1}{k_1 e^{M \beta_0}}.
\]
So we choose $k_1>0$ small such that assumption (\ref{eqn:fpassumption1}) is satisfied and then $\sigma$ is determined by the condition $f(\sigma) = M$. 

In summary, the positive constants $M$, $\beta_0$, $k_1$ and $\sigma$ are given by
\begin{equation}
\left\{
\begin{array}{l}
\medskip
M = \sup_\Omega \abs{u - u_0}\leq 2C_1 \\
\medskip
\alpha_1 = 2C_1 + \sup_{[0, D/2]} \abs{a_1(x)} \\
\medskip
\alpha_2 = 2C_1^2 \tan(D/2) + C_1 \sup_{[0,D/2]} \abs{a_2(x)} \\
\medskip
\beta_0 = 4 \max\set{\alpha_1, \alpha_2} + \sup_{[0, D/2]}\abs{u''_0}
\end{array}
\right.
\end{equation}
and
\begin{equation}\label{eqn:k1sigmaupperbarrier}
\left\{
\begin{array}{l}
k_1 = \min \set{\frac{1}{4}e^{-M \beta_0}, \dfrac{e^{-M \beta_0}}{2\sup_{[0, D/2]}\abs{u_0'}}}, \medskip\\
\sigma = \min\set{\dfrac{k_1}{\beta_0}\pset{e^{M\beta_0}-1}, D/2}. \medskip
\end{array}
\right.
\end{equation}
For $w(x) = u_0(x) + f(x)$ with 
\[
f(x) = \frac{1}{\beta_0} \log \pset{1+ \frac{\beta_0}{k_1}x} \quad \text{for} \quad x\in [0, \sigma],
\]
we have
\begin{enumerate}
\item $\bar P w < 0$ and $\bar P(u) = 0$ on $(0,\sigma)\times (0, T)$,
\item $w(x) \geq u_0(x)$ for $x\in [0, \sigma]$,
\item $w\geq u$ on the parabolic boundary of $(0, \sigma) \times(0, T)$,
\item $w(0) = u_0(0)$.
\end{enumerate}
The weak maximum principle of the operator $\bar P$ implies that $w(x) \geq u(x, t)$ for any $(x, t)\in (0, \sigma) \times(0, T)$. This gives the upper barrier $w^+ = w$ as needed for the boundary gradient estimate at $x = 0$. Note that $w^-(x) = u_0(x) - f(x)$ gives the lower barrier. Then the rest of the proof for the boundary $x=0$ follows from \cite[pp. 232-233]{Lieberman}, i.e., we have
\[
\sup_{(y,s)\in \Omega} \frac{\abs{u(0, t) - u(y,s)}}{\abs{(0,t) - (y,s)}} \leq \max\set{L^+, L^-, \frac{M}{\sigma}}
\]
where $L^{\pm}$ are Lipschitz constants of $w^{\pm}$ respectively.

The argument for the boundary $x = D/2$ is similar, and thus we finishes the proof. 
\end{proof}

\begin{rem}\label{rem:boundarygradient}
Lemma \ref{lem:boundarygradestimate} also holds when the initial data $u_0$ is given by the piecewise smooth function $\psi_{k,0}-\tilde\psi$. In fact we choose $\sigma$ in (\ref{eqn:k1sigmaupperbarrier}) not exceeding $\kappa$, where $u_0$ is smooth on $[0, \kappa]$ and $[D/2 - \kappa, D/2]$. Then we have
\begin{equation}\label{eqn:boundarygradientcont}
\sup_{\begin{array}{c}(x,t) \in S\Omega \\ (y,s)\in \Omega \end{array}} \frac{\abs{u(x, t) - u(y,s)}}{\abs{(x,t)-(y,s)}} \leq \tilde C_2
\end{equation}
where $\tilde C_2$ depends on $C_1$ in Lemma \ref{lem:Maxestimate}, the operator $P$, $\sup_{[0, \kappa]}\abs{u_0'}$, $\sup_{[0, \kappa]} \abs{u_0''}$, $\sup_{[D/2-\kappa,D/2]}\abs{u_0'}$, $\sup_{[D/2 - \kappa, D/2]} \abs{u_0''}$, and $\kappa$. 
\end{rem}

Now we give the proof when the initial condition $u_0$ is in the H\"older space $C^{2+\alpha}(\POmega)$.

\begin{proof}[Proof of Theorem \ref{thm:uPDEHolder}] 
The existence of a solution in $C^{2+\alpha}(\Omega)$ follows from Theorems 8.2 and 8.3 in \cite{Lieberman}. First we show the compatibility condition at $X=(0,0)$ and $(D/2,0)$. Since $u(0, t) = 0$ and $u(D/2,t) = 0$ are constants, we have $u_t = 0$ at $X = (0, 0)$ and $(D/2, 0)$. So $P(u) = 0$ at these two points is equivalent to the condition that $u_0(x)$ is a stationary solution near $x = 0$ and $D/2$, which is satisfied by our choice of $u_0$.

Next we derived the a prior estimate $\abs{u}_{1+\gamma}$ for some $\gamma > 0$. Suppose $u$ is a classical solution. Let $\beta_0, \beta_1$ be positive constants  such that 
\begin{equation*}
\abs{a(X, q, p)} \leq \beta_0 p^2 \quad \text{for } \abs{p} \geq \beta_1 \text{ and } (X, q)\in \Omega \times [-C_1, C_1].
\end{equation*}
Then from \cite[Theorem 11.16]{Lieberman}, we have the following global gradient estimate:
\begin{equation}\label{eqn:globalgradestimate}
\sup_\Omega \abs{u_x} \leq \pset{C_2 + \beta_1} \exp\pset{2\beta_0 C_1}
\end{equation}
where $C_1$ and $C_2$ are given in Lemmas \ref{lem:Maxestimate} and \ref{lem:boundarygradestimate}. Now for the global H\"older gradient estimate. Let $\mu_K>0$ be a constant such that 
\begin{equation*}
\abs{a(X, q, p)}\leq \mu_K, \text{ for } X\in \Omega \text{ and } \abs{q}+\abs{p} \leq C_1 + (C_2 + \beta_1)\exp(2\beta_0 C_1).
\end{equation*}
Theorem 12.10 in \cite{Lieberman} implies that, there are positive constants $\gamma$ and $C_3$ determined only by $\alpha$ and $\diam \Omega$, such that 
\begin{equation}\label{eqn:globalHoldergradestimate}
[u_x]_\gamma \leq C_3 \pset{C_1 + (C_2 + \beta_1)\exp(2\beta_0 C_1) +\mu_K + \abs{u_0}_{1+\alpha}}
\end{equation}
These two estimates in (\ref{eqn:globalgradestimate}) and (\ref{eqn:globalHoldergradestimate}) with the maximum estimate in Lemma \ref{lem:Maxestimate} yield
\[
\abs{u}_{1+\gamma} =\sup \abs{u} + \sup \abs{u_x} + [u_x]_{\gamma} \leq M_\gamma
\] 
where 
\[
M_\gamma = C_1 + \pset{C_2 +\beta_1}\exp(2\beta_0 C_1) + C_3 \pset{C_1 + (C_2 + \beta_1)\exp(2\beta_0 C_1) +\mu_K + \abs{u_0}_{1+\alpha}}.
\]
So Theorems 8.2 and 8.3 in \cite{Lieberman} implies that, there exists a solution $u \in C^{2+\alpha}(\Omega)$. Note that $u$ satisfies the linear equation $L u = 0$ with
\begin{equation}\label{eqn:Lbc}
L = - \pd_t + \pd^2_{x} + b(x, u) \pd_x + c(x, u)
\end{equation}
and
\begin{eqnarray*}
b(x, u) & = & 2\tilde\psi(x) - (n+1)\tan(x) + 2u(x, t) \\
c(x, u) & = & 2\tilde\psi'(x) - 4\tan(x) \tilde\psi(x) - (n-1)(1-\tan^2(x)) -2 \tan(x) u(x,t).
\end{eqnarray*}
The maximum estimate in Lemma \ref{lem:Maxestimate} and the global gradient estimate in (\ref{eqn:globalgradestimate}) show that $\abs{u}$ and $\abs{u_x}$ are bounded on $\Omega$. Then both $b(x,u)$ and $c(x,u)$ have bounded norm in $C^\alpha(\Omega)$. So the linear theory, for example, \cite[Theorem 5.14]{Lieberman} implies that $u \in C^{2,1}(\Omega) \cap C(\bar\Omega)$. This finishes the proof.
\end{proof}

\medskip

\section{A heat kernel proof of Lemma \ref{lem:psikt<0} }
\begin{center} by Qi S. Zhang 
	\end{center}

Department of Mathematics, University of California, Riverside, CA, 92521

\medskip

We follow the notions in Section 4 and 5. 

\begin{proof}
Let $G$ be the fundamental solution with Dirichlet boundary condition of the linear operator $L$ in (\ref{eqn:Lw}) on $[0, D/2]$. For $t > s > 0$, 
\begin{equation}\label{eqn:wintegralts}
w(z, t) = \int_0^{D/2} G(z, t, \xi, s) w(\xi, s) d\xi 
\end{equation}
solves the linear equation $L w = 0$ with Dirichlet boundary condition and initial data $w(z, s)$ at $t = s$. 

We claim that the statement in this lemma follows from the inequality $w\leq 0$ on $[0, D/2]\times (0, \infty)$. Assume that $z\in (0, D/2)$ and $t > 0$.  Take a time $s$ with $0<s< t$. Since $G(z, t, \xi, s) > 0$ when $\xi \in (0, D/2)$, it follows that $w(z, t) < 0$, unless $w(\cdot, s)$ is the trivial function on $[0, D/2]$. In the latter case, we have $\psi_{k}(\cdot, s) = \tilde \psi_{k,0}(\cdot)$ for any $s\in (0, t)$ which contradicts the fact the $\psi_k$ has the initial data $\psi_{k,0}$.    

To finish the proof, we need to show that $w\leq 0$. Recall the parabolic equation (\ref{eqn:uPDE}) of $u(z, t)$, and then  the integral formula (\ref{eqn:wintegralts}) yields 
\begin{eqnarray*}
w(z, t) & = & \int_0^{D/2} G(z, t, \xi, s)\big{\{}u''(\xi, s) + 2u'(\xi, s) u(\xi, s) + a_1(\xi) u'(\xi, s) \\
& & \qquad \qquad- 2\tan(\xi) u^2(\xi, s) + a_2(\xi) u(\xi, s) \big{\}} d\xi \\
& = & \int_0^{D/2}  \big{\{}- \pd_\xi G(z, t, \xi, s) u'(\xi, s) - \pd_\xi G(z, t, \xi, s) u^2(\xi, s) \\
& & \qquad \qquad - \pd_\xi \sqset{G(z, t, \xi, s)a_1(\xi)} u(\xi, s) \\
& & \qquad \qquad + G(z, t, \xi, s)\sqset{-2\tan(\xi) u^2(\xi, s) + a_2(\xi) u(\xi, s)}\big{\}} d\xi \\
& = & \int_0^{D/2} \big{\{} \pd^2_\xi G(z, t, \xi, s) u(\xi, s) - \pd_\xi G(z, t, \xi, s) u^2(\xi, s) \\
& & \qquad \qquad - \pd_\xi \sqset{G(z, t, \xi, s)a_1(\xi)} u(\xi, s) \\
& & \qquad \qquad + G(z, t, \xi, s)\sqset{-2\tan(\xi) u^2(\xi, s) + a_2(\xi) u(\xi, s)}\big{\}} d\xi.
\end{eqnarray*} 
The integration-by-parts in the second and third equalities above are justified by the facts that both $G(x, t, \xi, s)$ and $u(\xi, s)$ vanish at $\xi = 0$ and $D/2$, and that both $u(\xi, s)$ and $u'(\xi, s)$ are bounded continuous functions on $(0, D/2)$. It follows that we have
\begin{eqnarray}
w(z, t) & = & \int_0^{D/2}  \big{\{} \pd^2_\xi G(z, t, \xi, s) u(\xi, s) - \pd_\xi G(z, t, \xi, s) \sqset{u^2(\xi, s) + a_1(\xi) u(\xi, s)} \nonumber \\
& &  + G(z, t, \xi, s)\sqset{-2\tan(\xi) u^2(\xi, s) + a_2(\xi) u(\xi, s) - a_1'(\xi) u(\xi, s)}\big{\}} d\xi \label{eqn:wGukintegral}.
\end{eqnarray}
In equation (\ref{eqn:wGukintegral}), fix $t> 0$ and let $s = s_j \To 0$ as $j \To \infty$. Since $G(z, t, \xi, s)$, $\pd_\xi G(z, t, \xi, s)$, $\pd^2_\xi G(z, t, \xi, s)$ are bounded when $t - s \geq \eps_0 > 0$, see e.g. \cite[Section 7]{Aronson}, the other functions  including $u(\xi, t)$ in the integral of equation (\ref{eqn:wGukintegral}) are bounded, and $u\in C\pset{[0, D/2]\times [0, \infty)}$, the Dominated Convergence Theorem implies that 
\begin{eqnarray*}
w(z,t) & = & \int_0^{D/2} \left\{ \eta''(\xi) u(\xi, 0) - \eta'(\xi) \sqset{u^2(\xi, 0) + a_1(\xi) u(\xi, 0)} \right. \\
& & \left. + \eta(\xi) \sqset{-2\tan(\xi) u^2(\xi, 0) + a_2(\xi) u(\xi, 0) - a_1'(\xi) u(\xi, 0)} \right\} d\xi \\
& \leq & 0,
\end{eqnarray*}
where $\eta(\xi) = G(x, t, \xi, 0)\geq 0$. The last inequality above follows from the fact that the initial data $u(z, 0) = \psi_{k,0}(z) - \tilde \psi_{k,0}(z)$ is a supersolution to the parabolic equation (\ref{eqn:uPDE}), and treating $\eta(\xi)$ as a test function. This finishes the proof. 
\end{proof}

\medskip


\end{document}